\documentclass[11pt]{amsart}
\usepackage[a4paper,left=3cm,right=3cm]{geometry}
\usepackage[latin1]{inputenc}
\usepackage[english]{babel}
\usepackage{amsmath}
\usepackage{verbatim}
\usepackage{amssymb}
\usepackage{color}
\usepackage{hyperref}
\usepackage{amsthm}
\usepackage{tikz}
\usepackage{pgfplots}
\usepackage{tikz-3dplot}
\usepackage{float}
\usepackage{amsfonts}

\newtheorem{theorem}{Theorem}[section]

\newtheorem{lemma}[theorem]{Lemma}

\newcommand{\Om}{\Omega}

\newcommand{\sq}{\subseteq}

\newcommand{\vps}{\varepsilon}
\newcommand{\ra}{\rightarrow}
\newcommand{\RR}{\mathbb R}

\begin{document}
\title{Asymptotic behaviour of the Steklov problem on dumbbell domains}
\author[]{Dorin Bucur, Antoine Henrot, Marco Michetti}

\address[Dorin Bucur]{Univ. Savoie Mont Blanc, CNRS, LAMA \\
73000 Chambéry, France
}
\email{dorin.bucur@univ-savoie.fr}
\address[Antoine Henrot]{
Institut Elie Cartan de Lorraine \\ CNRS UMR 7502 and Universit\'e de Lorraine \\
BP 70239
54506 Vandoeuvre-l\`es-Nancy, France}
\email{antoine.henrot@univ-lorraine.fr}
\address[Marco Michetti]{
Institut Elie Cartan de Lorraine \\ CNRS UMR 7502 and Universit\'e de Lorraine \\
BP 70239
54506 Vandoeuvre-l\`es-Nancy, France}
\email{marco.michetti@univ-lorraine.fr}

\date{}

\begin{abstract}
    We analyse the asymptotic behaviour of the eigenvalues and eigenvectors of a Steklov problem in a dumbbell domain consisting of two Lipschitz sets connected by a thin tube with vanishing width. All the eigenvalues are collapsing to zero, the speed being driven by some power of the width which multiplies the eigenvalues of a one dimensional problem. In two dimensions of the space, the behaviour is fundamentally different from the  third or higher  dimensions and the limit problems are of different nature.  This phenomenon is due to the fact that only in dimension two the boundary of the tube has not vanishing surface measure. 
    \end{abstract}

\maketitle

\section{Introduction}
The purpose of this paper is to analyse the asymptotic behaviour of the eigenvalues and eigenfunctions of the Steklov problem in a dumbbell domain. Given $\Om \sq \RR^n$, open, bounded, connected, Lipschitz set, the Steklov problem on $\Om$ consists in solving the eigenvalue problem 
\begin{equation}
\begin{cases}
     \Delta u=0\qquad  \Omega   \\
      \partial_{\nu} u=\sigma u \qquad \partial\Omega,
\end{cases}
\end{equation}
where $\nu$ stands for the outward normal at the boundary.
As the trace operator $H^1(\Om) \ra L^2(\partial \Om)$ is compact, the spectrum of the Steklov problem is discrete and the eigenvalues (counted with their multiplicities) go to infinity
$$0= \sigma_0(\Om) <  \sigma_1(\Om)\le  \sigma_2 (\Om) \le \cdots \ra +\infty.$$
We also have the following variational characterization of the Steklov eignevalues
\begin{equation*}
\sigma_k(\Om)=\inf_{E_k} \sup_{0\neq u\in E_k } \frac{\int_{\Omega}|\nabla u|^2dx}{\int_{\partial \Omega}u^2 d{\mathcal H}^{n-1}},
\end{equation*}
where the infimum is taken over all $k$-dimensional subspaces of the Sobolev space $H^1(\Omega)$ which are  $L^2$-orthogonal to constants on $\partial\Omega$.

Let $\Omega_{\epsilon}\subset \mathbb{R}^n$ be a dumbbell shape domain
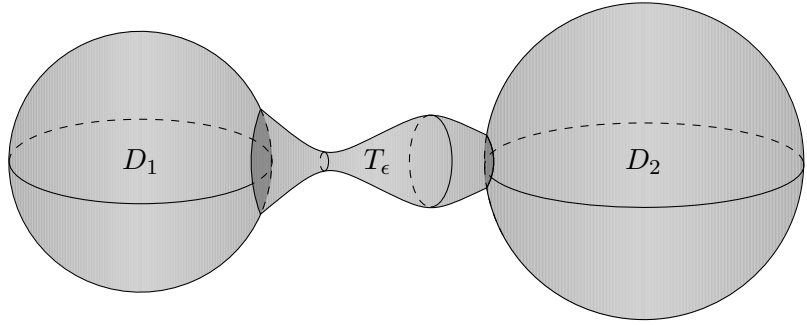
\begin{figure}
\centering
\tdplotsetmaincoords{0}{0}
\begin{tikzpicture}[tdplot_main_coords, scale=0.7]
\fill[left color=gray!50!black,right color=gray!50!black,middle color=gray!50,shading=axis,opacity=0.15] (-4.45,0,0) circle (2.47); 
\fill[left color=gray!50!black,right color=gray!50!black,middle color=gray!50,shading=axis,opacity=0.15] (5,0,0) circle (3);
\fill[left color=gray!50!black,right color=gray!50!black,middle color=gray!50,shading=axis,opacity=0.15] (-2.20,1,0) .. controls (-1,0,0,).. (0,0.5,0).. controls (1,1,0,)..  (2.05,0.5,0)  arc (30:-30:1).. controls (1,-1,0,).. (0,-0.5,0) .. controls (-1,0,0,).. (-2.20,-1,0) arc (200:160:2.9);
\draw (2.05,0.5,0)  arc (30:-30:1);
\draw (2.05,0.5,0) arc (170.5:-170.5:3);
\draw [dashed] (2.05,0.5,0) arc (170.5:210:3);

\draw (8,0,0) arc (5:-163:3 and 0.8);
\draw [dashed] (8,0,0) arc (5:197:3 and 0.8);
\draw (-2.20,1,0) .. controls (-1,0,0,).. (0,0.5,0).. controls (1,1,0,)..  (2.05,0.5,0)  arc (30:-30:1).. controls (1,-1,0,).. (0,-0.5,0) .. controls (-1,0,0,).. (-2.20,-1,0) arc (200:160:2.9);
\draw (-2.20,1,0) arc(24:337:2.47);
\draw [dashed ] (-2.20,1,0) arc(24:-24:2.47); 
\draw (-6.92,0,0) arc(180:328:2.47 and 0.8);
\draw [dashed] (-6.92,0,0) arc(180:-32:2.47 and 0.8);

\draw (1,0.88,0,) arc(90:-90:0.4 and 0.87);
\draw [dashed](1,0.88,0,) arc(90:270:0.4 and 0.87);
\draw (-1,0.18,0,) arc(90:-90:0.08 and 0.18);
\draw [dashed] (-1,0.18,0,) arc(90:270:0.08 and 0.18);
\draw (-4.45,0,0) node {$D_1$};
\draw (5,0,0) node {$D_2$};
\draw (0,0,0) node {$T_{\epsilon}$};
\end{tikzpicture}
\caption{Dumbbell shape domain $\Omega_{\epsilon}$.}
\label{fig1}
\end{figure}
given by (see  Figure \ref{fig1})

\begin{equation*}
 \Omega_{\epsilon}=D_1\cup T_{\epsilon}\cup D_2, 
\end{equation*}
where $D_1$ and $D_2$ are disjoint, bounded, open, connected sets  in $\mathbb{R}^n$ with Lipschitz boundary and $T_{\epsilon}$ is expressed as
\begin{equation*}
T_{\epsilon}=\big\{x=(x_1,x')\in \mathbb{R}^n | -\frac{L}{2}\leq x_1\leq \frac{L}{2}, |x'|< \epsilon\rho(x_1) \big\},
\end{equation*}
where $L>0$ and $\rho\in C^0([-\frac{L}{2},\frac{L}{2}])\cap C^{\infty}((-\frac{L}{2},\frac{L}{2}))$ is a positive function.

The connection between the channel and the two regions $D_1$ and $D_2$ occurs as follows: we
assume that there exist an orthogonal system of coordinate $x=(x_1,x_2,...,x_n)=(x_1,x')\in \mathbb{R}^n$ and two constants $L,\delta \in \mathbb{R}$ such that
\begin{align*}
\overline{D_1}\cap \big\{x=(x_1,x')\in \mathbb{R}^n | x_1\geq -\frac{L}{2}, |x'|\leq \delta\big\}&=\big\{x=(-\frac{L}{2},x')\in \mathbb{R}^n |\,\, |x'|\leq \delta\big\} \\
\overline{D_2}\cap \big\{x=(x_1,x')\in \mathbb{R}^n | x_1\leq \frac{L}{2}, |x'|\leq \delta\big\}&=\big\{x=(\frac{L}{2},x')\in \mathbb{R}^n |\,\, |x'|\leq \delta\big\}.
\end{align*}  

The eigenvalues of the Steklov problem in $\Om_\epsilon$ are denoted by
\begin{equation*}
0=\sigma_0^{\epsilon}<\sigma_1^{\epsilon}\leq \sigma_2^{\epsilon}\leq... \nearrow \infty \qquad \forall \epsilon>0,
\end{equation*}
multiplicity being counted, and the corresponding eigenfunctions by  $u_k^{\epsilon}$, which are normalized in $L^2(\partial \Om_\epsilon)$, $||u^{\epsilon}_k||_{L^2(\partial\Omega_{\epsilon})}=1$.

The main purpose of this work is to study what is the behaviour of $(\sigma_k^{\epsilon},u^{\epsilon}_k)$ when $\epsilon$ goes to $0$. The first thing to notice is that, if $\epsilon\rightarrow 0$, the channel $T_{\epsilon}$ collapse to a line and the norm of the trace operator blows up.  One can easily observe  that

\begin{equation*}
\forall k\in \mathbb{N}, \quad \sigma_k^{\epsilon}\rightarrow 0 \qquad \mbox{when } \epsilon \ra 0,
\end{equation*}
our objective being to give precise estimates of the asymptotic behaviour of $\sigma_k^{\epsilon}$ when $\epsilon$ goes to $0$. We shall prove that $\sigma_k^{\epsilon}$ behaves, roughly speaking, as $\mu_k \epsilon ^\gamma$, where $\mu_k$ is the $k$-th eigenvalue of some one dimensional problem and $\gamma \in \{1, n-1\}$. 

As an interesting feature, we notice that the behaviour strongly depends on the dimension of the ambient space. Indeed we have to distinguish between the cases $n=2$ and $n\geq 3$, as we shall see below. This fact is due to the presence of the boundary energy in the Rayleigh quotient of the Steklov problem and to the fact that in dimension three, or higher, the surface area measure of the boundary of the tube is vanishing with $\epsilon$.

Below, we denote by $P(D)$ the surface area measure of the boundary of $D$ and  $\omega_{n}$ is the Lebesgue measure of the $n-$dimensional unit ball. Let $\Phi_\vps : T_1 \ra T_\epsilon$, $\Phi_\epsilon (x_1, x')= (x_1, \epsilon x')$. 

Here are our main results. The first theorem concerns the case $n=2$. \begin{theorem}\label{th1}{\bf ($n=2$)}
Let $\Omega_{\epsilon}\subset \mathbb{R}^2$ be the dumbbell shape domain defined as above. Then
\begin{equation*}
\sigma_k^{\epsilon}\sim \mu_k \epsilon+o(\epsilon) \quad \text{as} \quad \epsilon\rightarrow 0, 
\end{equation*}
where $\mu_k$ is the $k-$th eigenvalue  of the following problem 
\begin{equation}\label{ep1}
\begin{cases}
\vspace{0.2cm}
   -\frac{d}{dx}\big(\rho(x)\frac{d V_k}{dx}(x)\big)=\mu_k V_k(x)  \qquad  x\in \big(-\frac{L}{2},\frac{L}{2} \big) \\
\vspace{0.2cm}
      \rho(-\frac{L}{2})\frac{dV_k}{dx}(-\frac{L}{2})=-\frac{\mu_k}{2}P(D_1)V_k(-\frac{L}{2}) \\
      \rho(\frac{L}{2})\frac{dV_k}{dx}(\frac{L}{2})=\frac{\mu_k}{2}P(D_2)V_k(\frac{L}{2}).
\end{cases}
\end{equation}
For every subsequence $\{ \epsilon_n\}_{n=1}^{\infty}$ such that $\epsilon_n\rightarrow 0$, we have 
\begin{equation*}
u_k^{\epsilon_n}\circ \Phi_{\epsilon_n} \rightharpoonup \overline{V}_k \quad \text{in} \quad H^1(T_1),
\end{equation*}
where $\overline{V}_k$ is a  $k-$th eigenfunction of the problem \eqref{ep1} constantly extended in the variable $x_2$. 
\end{theorem}
This kind of eigenvalue problem (in any dimension)
where the eigenvalue $\mu_k$ appears both inside the domain and
in the boundary condition is sometimes called a {\it dynamical eigenvalue problem}. It
appears at different places in the literature. we refer for example to \cite{vBF} where a complete
study of this eigenvalue problem has been done.
See also \cite{GHL} where a similar problem appears in the homogenization of the Steklov problem.

\medskip
The next two theorems concern the case $n\geq 3$. We shall distinguish between the behaviour of the first non-zero eigenvalue, and the others.
\begin{theorem}\label{th2}{\bf ($n\ge 3$,  $k \ge 2$)}
Let $\Omega_{\epsilon}\subset \mathbb{R}^n$ be the dumbbell shape domain defined as above and $n\geq 3$. Then for all $k\geq 2$ we have 
\begin{equation*}
\sigma_k^{\epsilon}\sim \alpha_{k-1} \epsilon+o(\epsilon) \quad \text{as} \quad \epsilon\rightarrow 0, 
\end{equation*}
where $\alpha_{k-1}$is the $(k-1)-$th eigenvalue (counting from zero) of
\begin{equation}\label{ep3}
\begin{cases}
\vspace{0.2cm}
   -w_{n-1}\frac{d}{dx}\big(\rho^{n-1}(x)\frac{d V_k}{dx}(x)\big)=\alpha_kw_{n-2}\rho^{n-2}(x)V_k(x)  \qquad  x\in \big(-\frac{L}{2},\frac{L}{2} \big) \\
   \vspace{0.2cm}
   V_k(-\frac{L}{2})=0 \\
   V_k(\frac{L}{2})=0.
\end{cases}
\end{equation}
For every subsequence $\{ \epsilon_n\}_{n=1}^{\infty}$ such that $\epsilon_n\rightarrow 0$, we have 
\begin{equation*}
\epsilon_n^{\frac{n-2}{2}}u_k^{\epsilon_n}\circ \Phi_{\epsilon_n} \rightharpoonup \overline{V}_{k-1}\quad \text{in} \quad H^1(T_1),
\end{equation*}
where $\overline{V}_{k-1}$ is an eigenfunction corresponding to $\alpha_{k-1}$, constantly extended into the variables $x_i$ for $2\leq i\leq n$.
\end{theorem}
Therefore, in the case $n\ge 3$,  $k \ge 2$ we end up with a classical Dirichlet eigenvalue problem.
\begin{theorem}\label{th3}{\bf ($n\ge 3$,  $k =1$)}
Let $\Omega_{\epsilon}\subset \mathbb{R}^n$ be the dumbbell shape domain defined as above and $n\geq 3$. The first Steklov eigenvalue has the following asymptotic behaviour
\begin{equation*}
\sigma_1^{\epsilon}\sim \sigma_1 \epsilon^{n-1}+o(\epsilon^{n-1}) \quad \text{as} \quad \epsilon\rightarrow 0, 
\end{equation*}
where $\sigma_1$ is the unique positive number such that the following differential equation has a 
non-trivial solution:
\begin{equation}\label{ep2}
\begin{cases}
\vspace{0.2cm}
   -\omega_{n-1}\frac{d}{dx}\big(\rho^{n-1}(x)\frac{d V_1}{dx}(x)\big)=0  \qquad  x\in \big(-\frac{L}{2},\frac{L}{2} \big) \\
   \vspace{0.2cm}
   \rho^{n-1}(-\frac{L}{2})\frac{dV_1}{dx}(-\frac{L}{2})=-\frac{\sigma_1}{\omega_{n-1}}P(D_1)V_1(-\frac{L}{2}) \\
   \rho^{n-1}(\frac{L}{2})\frac{dV_1}{dx}(\frac{L}{2})=\frac{\sigma_1}{\omega_{n-1}}P(D_2)V_1(\frac{L}{2}).
\end{cases}
\end{equation}
For every subsequence $\{ \epsilon_n\}_{n=1}^{\infty}$ such that $\epsilon_n\rightarrow 0$, we have 
\begin{equation*}
u_1^{\epsilon_n}\circ \Phi_{\epsilon_n} \rightharpoonup \overline{V}_1 \quad \text{in} \quad H^1(T_1),
\end{equation*}
where $\overline{V}_1$ is the solution of the equation \eqref{ep2} constantly extended to the variables $x_i$ for $2\leq i\leq n$. 
\end{theorem}

Let us now comment on the existing literature. A similar problem for the eigenvalues of the Neumann Laplacian has been deeply studied, in particular in a series of papers by S. Jimbo. 
A first characterization of the eigenvalues in the Neumann case was given in \cite{Be73}. In \cite{Ji89} there is a complete description of the behaviour of the Neumann eigenfunctions and in \cite{JM92} there is a complete description of the Neumann eigenvalues when the channel collapse to a segment. 
Other references for the Neumann problem in dumbbell shape domains are \cite{Ar95,Ji188,Ji288,Ji93,Ji04}. These results turn out to be very useful in the study of the solutions of reaction diffusion systems in singular domains (see for instance \cite{ACL06,Fa90,HV84,Mo90}). 
Perturbations of the geometric domain for the Steklov problem have been considered in \cite{GP10}. For an asymptotic behaviour  of the Steklov problem on a singular perturbation somehow close to our analysis, we refer to the result of Nazarov \cite{Na12} where he studies a two dimensional domain obtained by the junction of two rectangles (see also \cite{Na14} for a perturbation by a small whole). 
At last, let us mention that in the case of Dirichlet boundary conditions, singular perturbations of this type are less interesting, since the spectrum is stable to this geometric perturbation.  Indeed, it can be proved that the dumbbell $\gamma$-converges to the union of the two sets $D_1\cup D_2$ which
means that its Dirichlet eigenvalues converge to the union of the spectrum of $D_1$ and $D_2$. We refer
to the books \cite{BuBu} and \cite{HPi} for more details.

\section{The case $n=2$. Proof of Theorem \ref{th1}.}
In this section we will prove Theorem \ref{th1}. 
We define $\partial T_{\epsilon}^e\subset \partial \Omega_{\epsilon}$ in the following way 

\begin{equation*}
\partial T_{\epsilon}^e=\big\{x=(x_1,x')\in \mathbb{R}^n | -\frac{L}{2}\leq x_1\leq \frac{L}{2}, x'=\epsilon|\rho(x_1)| \big\}
\end{equation*}

In two dimensions, the set  $\partial T_{\epsilon}^e$ is not connected and we decompose it 
\begin{equation*}
\partial T_{\epsilon}^e=\Gamma_{\epsilon}^-\cup\Gamma_{\epsilon}^+,
\end{equation*}
where 
\begin{align*}
\Gamma_{\epsilon}^+&=\big\{x=(x_1,x_2)\in \mathbb{R}^n | -\frac{L}{2}\leq x_1\leq \frac{L}{2}, x_2=\epsilon\rho(x_1) \big\}\\
\Gamma_{\epsilon}^-&=\big\{x=(x_1,x_2)\in \mathbb{R}^n | -\frac{L}{2}\leq x_1\leq \frac{L}{2}, x_2=-\epsilon\rho(x_1) \big\}.
\end{align*}

\subsection{Upper bound for Steklov eigenvalues}\label{sec3.1}
First of all we prove that there exists a constant $C>0$ such that the following upper bound holds for $\epsilon $ small enough:
\begin{equation}\label{stekbound1}
\sigma_k^{\epsilon}\leq C \epsilon.
\end{equation}
Precisely, we prove the following.
\begin{lemma} Let $\mu_k$ the $k-$th eigenvalue of \eqref{ep1} then we have 
\begin{equation}\label{stekin1}
\sigma_k^{\epsilon}\leq \mu_k\epsilon+o(\epsilon).
\end{equation}
\end{lemma}
\begin{proof}
In order to obtain this upper bound, we use the variational formulation  
\begin{equation*}
\sigma_k^{\epsilon}=\inf_{E_k} \sup_{0\neq u\in E_k } \frac{\int_{\Omega_{\epsilon}}|\nabla u|^2dx}{\int_{\partial \Omega_{\epsilon}}u^2 ds},
\end{equation*}
where the infimum is taken over all $k-$dimensional subspace of the Sobolev space $H^1(\Omega_{\epsilon})$ which are orthogonal to constants on $\partial\Omega_{\epsilon}$. We   choose   a particular subspace $E_k$ in order to obtain the upper bound.

We consider the eigenvalue problem \eqref{ep1} and take a basis of eigenfunctions $\{\phi_i\}_{i\in \mathbb{N}}$ normalized in the following way 
\begin{equation}\label{denn1}
\int_{\frac{-L}{2}}^{\frac{L}{2}}\phi_i \phi_j dx_1+\frac{1}{2}P(D_2)\phi_i\big (\frac{L}{2}\big )  \phi_j \big (\frac{L}{2}\big )+\frac{1}{2}P(D_1)\phi_i\big (-\frac{L}{2}\big ) \big (-\frac{L}{2}\big )=\delta_{ij}.
\end{equation}

Then
\begin{align}
\int_{\frac{-L}{2}}^{\frac{L}{2}}\rho \phi_i'\phi_j'dx_1&=0 \ \mbox{if $i\not= j$}\label{den2}\\
\int_{\frac{-L}{2}}^{\frac{L}{2}}\rho (\phi_i')^2dx_1&=\mu_i\label{den3}.
\end{align}
From the variational formulation of the eigenvalue problem we know that $\forall v\in H^1(-\frac{L}{2},\frac{L}{2})$ 
\begin{equation*}
\int_{\frac{-L}{2}}^{\frac{L}{2}}\rho \phi_i'v'dx_1=\frac{\mu_i}{2}P(D_2)\phi_i\big (\frac{L}{2}\big )v\big (\frac{L}{2}\big )+\frac{\mu_i}{2}P(D_1)\phi_i\big (-\frac{L}{2}\big )v\big (-\frac{L}{2}\big )+\mu_i\int_{\frac{-L}{2}}^{\frac{L}{2}}\phi_ivdx_1, 
\end{equation*}
we choose $v=1$ we obtain:
\begin{equation}\label{mean1}
\frac{1}{2}P(D_2)\phi_i\big (\frac{L}{2}\big )+\frac{1}{2}P(D_1)\phi_i\big (-\frac{L}{2}\big )+\int_{\frac{-L}{2}}^{\frac{L}{2}}\phi_idx_1=0.
\end{equation}

We now introduce our test functions that are the basis of our test subspace $E_k$. We define
\begin{equation*}
\Phi_i=\begin{cases}
\vspace{0.2cm}
  \phi_i(-\frac{L}{2}) \qquad \text{if}\,\, (x_1,x_2)\in D_1 \\
   \vspace{0.2cm}
  \phi_i(x_1) \qquad \,\,\, \text{if}\,\, (x_1,x_2)\in T_{\epsilon} \\
   \phi_i(\frac{L}{2}) \qquad \,\,\, \text{if}\,\, (x_1,x_2)\in D_2,
\end{cases}
\end{equation*}
and we introduce its mean value
\begin{equation*}
m_i^{\epsilon}=\frac{1}{|\partial \Omega_{\epsilon}|}\int_{\partial \Omega_{\epsilon}}\Phi_i ds.
\end{equation*}
The mean goes to zero if $\epsilon\rightarrow 0$, indeed
\begin{equation*}
\int_{\partial \Omega_{\epsilon}}\Phi_i ds=P(D_2)\phi_i\big (\frac{L}{2}\big )+P(D_1)\phi_i\big (-\frac{L}{2}\big )+2\int_{\frac{-L}{2}}^{\frac{L}{2}}\phi_i\sqrt{1+\epsilon^2\rho'^2}dx_1,
\end{equation*}
from equation \eqref{mean1}, dominated convergence and the fact that $|\partial \Omega_{\epsilon}|\rightarrow P(D_1)+P(D_2)+2L>0$ we obtain 
\begin{equation}\label{meanto0}
m_i^{\epsilon}\rightarrow 0 \quad \forall i\in \mathbb{N}.
\end{equation}

We introduce now our basis elements
\begin{equation*}
\Psi_i=\Phi_i-m_i^{\epsilon},
\end{equation*}
and our subspace will be $E_k=\text{Span}<\Psi_1,...,\Psi_k>$. Now we compute all the quantities we need for the Rayleigh quotient. We start by  the numerator, if $i\not= j$:
\begin{equation*}
\int_{\Omega_{\epsilon}}\nabla\Psi_i\cdot\nabla\Psi_j ds=\int_{T_{\epsilon}}\nabla\Phi_i\cdot\nabla\Phi_j ds=2\epsilon\int_{\frac{-L}{2}}^{\frac{L}{2}}\rho \phi_i'\phi_j'dx_1=0,
\end{equation*}
where the last equality is given by \eqref{den2}, and
\begin{equation*}
\int_{\Omega_{\epsilon}}|\nabla\Psi_i|^2ds=\int_{T_{\epsilon}}|\nabla\Phi_i|^2ds=2\epsilon\int_{\frac{-L}{2}}^{\frac{L}{2}}\rho (\phi_i')^2dx_1=2\epsilon\mu_i,
\end{equation*}
where the last equality is given by \eqref{den3}. Now we compute the terms in the denominator,
\begin{align*}
f_{i,j}(\epsilon)&:=\int_{\partial\Omega_{\epsilon}}\Psi_i\Psi_j ds=\int_{\partial\Omega_{\epsilon}}(\Phi_i-m_i^{\epsilon})(\Phi_j-m_j^{\epsilon})ds=
\int_{\partial\Omega_{\epsilon}}\Phi_i\Phi_j ds-m_i^{\epsilon}m_j^{\epsilon}P(\Omega_{\epsilon})\\
&=\frac{1}{2}P(D_2)\phi_i\big (\frac{L}{2}\big )\phi_j\big (\frac{L}{2}\big )+\frac{1}{2}P(D_1)\phi_i\big (-\frac{L}{2}\big )\phi_j\big (-\frac{L}{2}\big )\\
&\qquad \qquad  \qquad  \qquad \qquad \qquad  \qquad+\int_{\frac{-L}{2}}^{\frac{L}{2}}\phi_i\phi_j\sqrt{1+\epsilon^2\rho'^2}dx_1-m_i^{\epsilon}m_j^{\epsilon}P(\Omega_{\epsilon}).
\end{align*}
From \eqref{meanto0}, \eqref{denn1} and the dominated convergence we obtain
\begin{equation}\label{den4}
\lim_{\epsilon\rightarrow 0}f_{i,j}(\epsilon)=0 \quad\, i\neq j.
\end{equation}
Similarly,
\begin{align*}
f_{i,i}(\epsilon)&:=\int_{\partial\Omega_{\epsilon}}\Psi_i^2 ds=\int_{\partial\Omega_{\epsilon}}(\Phi_i-m_i^{\epsilon})^2ds= \int_{\partial\Omega_{\epsilon}}\Phi_i^2 ds-(m_i^{\epsilon})^2P(\Omega_{\epsilon})\\
&=2\int_{\frac{-L}{2}}^{\frac{L}{2}}\phi_i^2\sqrt{1+\epsilon^2\rho'^2}dx_1+P(D_2)\phi_i\big (\frac{L}{2}\big )^2+P(D_1)\phi_i\big (-\frac{L}{2}\big )^2-(m_i^{\epsilon})^2P(\Omega_{\epsilon}),
\end{align*}
now from \eqref{meanto0}, \eqref{denn1} and dominated convergence we obtain,
\begin{equation}\label{den5}
\lim_{\epsilon\rightarrow 0}f_{i,i}(\epsilon)=2.
\end{equation}
Now if we use the test subspace $E_k$ in the variational characterization we obtain
\begin{equation*}
\sigma_k^{\epsilon}\leq \sup_{(x_1,...,x_k)\in \mathbb{R}^k}\frac{2\epsilon\sum_{i=1}^kx_i^2\mu_i}{\sum_{i=1}^kx_i^2f_{i,i}(\epsilon)+\sum_{i<j}2x_ix_jf_{i,j}(\epsilon)},
\end{equation*}
if $\epsilon$ is small enough from \eqref{den4} and \eqref{den5} we obtain
\begin{equation}
\sigma_k^{\epsilon}\leq \sup_{(x_1,...,x_k)\in \mathbb{R}^k}\frac{\epsilon\sum_{i=1}^kx_i^2\mu_i}{\sum_{i=1}^kx_i^2}+o(\epsilon)=\mu_k\epsilon+o(\epsilon).
\end{equation}
\end{proof}
\subsection{Convergence of eigenfunctions}\label{sec3.2}
We start by showing the convergence on the two regions $D_i$ where $i=1,2$
\begin{lemma}\label{lemc1}
Let $k\geq 1$ we have (up to a sub-sequence that we still denote by $u_k^{\epsilon}$) 
\begin{align*}
u_k^{\epsilon} &\rightharpoonup c_{i,k} \quad \text{in} \quad H^1(D_i),\\
u_k^{\epsilon} &\rightarrow c_{i,k} \quad \text{locally uniformly in} \quad D_i.
\end{align*} 
where $c_{i,k}\in \RR$ are constants
\end{lemma}
\begin{proof}
First of all we know that $\sigma_k^{\epsilon}\rightarrow 0$ as $\epsilon$ goes to $0$, and from $||u_k^{\epsilon}||_{L^2(\partial\Omega_{\epsilon})}=1$ we conclude that
\begin{equation*}
\lim_{\epsilon\rightarrow 0}\int_{\Omega_{\epsilon}}|\nabla u_k^{\epsilon}|^2dx=0,
\end{equation*}
so it means that $||\nabla u_k^{\epsilon} ||_{L^2(D_1)}\leq ||\nabla u_k^{\epsilon}||_{L^2(\Omega_{\epsilon})}\leq C$. Now we want to bound $||u_k^{\epsilon}||_{L^2(D_1)}$ uniformly on $\epsilon$. Using Poincar\'e-Friedrichs inequality we obtain
\begin{equation*}
\int_{D_1} (u_k^{\epsilon})^2dx\leq \int_{\Omega_{\epsilon}} (u_k^{\epsilon})^2dx\leq C_{\Omega_{\epsilon}} \big [ \int_{\Omega_{\epsilon}} |\nabla u_k^{\epsilon}|^2dx+\int_{\partial\Omega_{\epsilon}} (u_k^{\epsilon})^2 ds \big ],
\end{equation*}
we know that $||u_k^{\epsilon}||_{L^2(\partial\Omega_{\epsilon})}=1$ and  $||\nabla u_k^{\epsilon}||_{L^2(\Omega_{\epsilon})}\leq C$, we have only to check that $C_{\Omega_{\epsilon}}\leq C\leq \infty$ if $\epsilon$ is small enough.

We have the following variational characterization for the constant $C_{\Omega_{\epsilon}}$
\begin{equation*}
\frac{1}{C_{\Omega_{\epsilon}}}=\inf_{v\in H^1(\Omega_{\epsilon})} \frac{\int_{\Omega_{\epsilon}} |\nabla v|^2dx+\int_{\partial\Omega_{\epsilon}} v^2ds}{\int_{\Omega_{\epsilon}} v^2dx}=\lambda_1(\Omega_{\epsilon},1),
\end{equation*}
 where $\lambda_1(\Omega_{\epsilon},1)$ is the first Robin eigenvalue with boundary parameter $1$ (see \cite{Da06}). We denote by $B_{R_{\epsilon}}$ the ball with the same measure of $\Omega_{\epsilon}$, now, using the Bossel-Daners inequality and the rescaling property of the Robin eigenvalue (see \cite{Da06}), we obtain
 \begin{equation*}
 \frac{1}{C_{\Omega_{\epsilon}}}=\lambda_1(\Omega_{\epsilon},1)\geq \lambda_1(B_{R_{\epsilon}},1)=\frac{1}{R_{\epsilon}^2}\lambda_1(B_1,R_{\epsilon}).
 \end{equation*}
Now, for $\epsilon$ small enough, we have $|D_1|+|D_2|\leq |\Omega_{\epsilon}|\leq |D_1|+|D_2|+1 $ so, by monotonicity of the Robin eigenvalue  on balls we finally obtain
 \begin{equation*}
 \frac{1}{C_{\Omega_{\epsilon}}}=\lambda_1(\Omega_{\epsilon},1)\geq \lambda_1(B_{R_{\epsilon}},1)=\frac{\pi}{|D_1|+|D_2|+1}\lambda_1\big (B_1,\sqrt{\frac{|D_1|+|D_2|}{\pi}}\big )>0.
 \end{equation*}
Finally we conclude that $C_{\Omega_{\epsilon}}\leq C<\infty$ for $\epsilon$ small enough. We conclude that
\begin{equation*}
||u_k^{\epsilon}||_{H^1(D_1)}\leq C<\infty,
\end{equation*}
so exist a sequence, that we still denote by $u_k^{\epsilon}$, and $u_k^0\in H^1(D_1)$ such that
\begin{equation*}
u_k^{\epsilon} \rightharpoonup u_k^0 \quad \text{in} \quad H^1(D_1).
\end{equation*}
 We also know that $||\nabla u_k^{\epsilon}||_{L^2(D_1)}\rightarrow 0$, so we conclude that there exists a constant $c_{1,k}\in \mathbb{R}$ such that
\begin{equation*}
u_k^{\epsilon} \rightharpoonup c_{1,k} \quad \text{in} \quad H^1(D_1).
\end{equation*}
We can improve this convergence since $u_k^{\epsilon}$ are harmonic. Fix a compact set $K\subset D_1$ and take $\delta >0$ such that $B_{\delta}(x)\subset D_1$ for all $x\in K$. By the average properties of harmonic functions and the Cauchy-Schwartz inequality we have:
\begin{equation*}
|u_k^{\epsilon}(x)| =\frac{1}{|B_{\delta}(x)|} \int_{B_{\delta}(x)} |u_k^{\epsilon}(y)|dy \leq |B_{\delta}(x)|^{-\frac{1}{2}}||u_k^{\epsilon}||_{L^2(D_1)}\leq C.
\end{equation*}
Up to a subsequence, we get that $u_k^{\epsilon}$ uniformly converges on $K$ to a constant. We conclude that for $i=1,2$
\begin{align*}
u_k^{\epsilon} &\rightharpoonup c_{i,k} \quad \text{in} \quad H^1(D_i),\\
u_k^{\epsilon} &\rightarrow c_{i,k} \quad \text{locally uniformly in} \quad D_i.
\end{align*}
\end{proof}

Now we study the behaviour of the eigenfunctions in the tube $T_{\epsilon}$. We define the following functions
\begin{equation*}
v_k^{\epsilon}(x_1,x_2)=u_k^{\epsilon}(x_1,\epsilon x_2)\quad \forall \, (x_1,x_2)\in T_1
\end{equation*}
\begin{lemma}\label{lemc2} Let $k\geq 1$. There exists $\overline{V}_k\in H^1(T_1)$ such that 
\begin{equation*}
v_k^{\epsilon} \rightharpoonup \overline{V}_k \quad \text{in} \quad H^1(T_1),
\end{equation*}
(up to a sub-sequence, still denoted by $v_k^{\epsilon}$), where $\overline{V}_k$ depends only on the variable $x_1$.
\end{lemma}
\begin{proof}
We start with the bound of $||\nabla v_k^{\epsilon}||_{L^2(T_1)}$ 
\begin{equation*}
\int_{T_1}|\nabla v_k^{\epsilon}|^2 dx\leq \int_{T_1} \Big( \frac{\partial v_k^{\epsilon}}{\partial x_1} \Big )^2+\frac{1}{\epsilon^2} \Big( \frac{\partial v_k^{\epsilon}}{\partial x_2} \Big )^2 dx=\frac{1}{\epsilon}\int_{T_{\epsilon}}|\nabla u_k^{\epsilon}|^2 dy\leq C
\end{equation*}
where we did the change of coordinates $y_1=x_1$, $y_2=\epsilon x_2$ and the last inequality is true because of \eqref{stekbound1}. We want now to bound $||v_k^{\epsilon}||_{L^2(T_1)}$. By the Poincar\'e-Friedrichs
inequality we get \begin{equation*}
\int_{T_1} (v^{\epsilon}_k)^2dx\leq C_{T_1} \big [ \int_{T_1} |\nabla v^{\epsilon}_k|^2dx+\int_{\Gamma_1^+\cup \Gamma_1^- } (v^{\epsilon}_k)^2ds \big ].
\end{equation*} 
Now $||\nabla v_k^{\epsilon}||_{L^2(T_1)}$ is bounded, so it remains to bound the second term in the r.h.s. of the inequality. Since $\sqrt{1+\rho'^2}$ is a bounded function, for $\epsilon$ small enough we obtain
\begin{align*}
\int_{\Gamma_1^+}(v^{\epsilon}_k)^2ds&=\int_{-\frac{L}{2}}^{\frac{L}{2}}(v_k^{\epsilon}(x_1,\rho(x_1))^2\sqrt{1+\rho'^2}dx_1\\
&\leq C\int_{-\frac{L}{2}}^{\frac{L}{2}}(u_k^{\epsilon}(x_1,\epsilon \rho(x_1))^2\sqrt{1+\epsilon^2\rho'^2}dx_1\\
&\leq C\int_{\Gamma_{\epsilon}^+}(u^{\epsilon}_k)^2ds\leq C
\end{align*}
where the last inequality is true becuase $||u_k^{\epsilon}||_{L^2(\partial\Omega_{\epsilon})}=1$. The same computation is true for the integral over $\Gamma_1^-$.

We conclude that there exists $\overline{V}_k\in H^1(T_1)$ such that (up to a sub-sequence that we still denote by $v_k^{\epsilon}$) 
\begin{equation*}
v_k^{\epsilon} \rightharpoonup \overline{V}_k \quad \text{in} \quad H^1(T_1).
\end{equation*}
We finish the proof by showing that $\overline{V}_k$ does not depend on $x_2$. Indeed 
\begin{equation*}
\int_{T_1} \big ( \frac{\partial v_k^{\epsilon} }{\partial x_2} \big )^2dx=\epsilon \int_{T_{\epsilon}} \big ( \frac{\partial u_k^{\epsilon} }{\partial x_2} \big )^2dx\leq C \epsilon^2 \rightarrow 0.
\end{equation*} 
\end{proof} 
\subsection{Limit eigenvalue problem}\label{sec3.3}
From \eqref{stekin1} we know that there exists $0\leq \beta \leq \mu_k$ such that
\begin{equation*}
\frac{\sigma_k^{\epsilon}}{\epsilon}\rightarrow \beta.
\end{equation*}
First,  we  prove that there exists $j\in \mathbb{N}$ such that $0\leq j\leq k$ and $\beta=\mu_j$ and, in a second step, we prove that $j=k$.

\noindent{\bf Step 1.} We begin with the following. 
\begin{lemma}\label{lep1} There exists $j\in \mathbb{N}$  such that $0\leq j\leq k$ and
\begin{equation*}
\frac{\sigma_k^{\epsilon}}{\epsilon}\rightarrow \mu_j,
\end{equation*}
where $\mu_j$ is the $j-$th eigenvalue of the problem \eqref{ep1}.
\end{lemma}
\begin{proof}
We define $\beta$ in such a way that
\begin{equation*}
\frac{\sigma_k^{\epsilon}}{\epsilon}\rightarrow \beta,
\end{equation*}
from \eqref{stekin1} we know that $0\leq \beta\leq \mu_k$. We use the variational formulation of the Steklov problem with the following test function $\phi\in C_c^{\infty}(-\frac{L}{2},\frac{L}{2})$ (we constantly extend $\phi$ in the last variable $x_2$), we obtain:
\begin{equation*}
\int_{T_{\epsilon}} \frac{\partial u_k^{\epsilon} }{\partial x_1} \frac{\partial \phi }{\partial x_1} dx=\sigma_k^{\epsilon} \int_{\Gamma_{\epsilon}^+\cup \Gamma_{\epsilon}^-}  u_k^{\epsilon} \phi ds.
\end{equation*}
Now we make the  change of variable $y_1=x_1$ and $y_2=\epsilon x_2$ and we write the integral in the right hand side by the integral in the graph of $\epsilon \rho$
\begin{align*}
\int_{T_1} \frac{\partial v_k^{\epsilon}}{\partial y_1} \frac{\partial \phi }{\partial y_1} dy&=\frac{\sigma_k^{\epsilon}}{\epsilon}\big ( \int_{-\frac{L}{2}}^{\frac{L}{2}}u_k^{\epsilon}(x_1,\epsilon\rho(x_1)) \phi\sqrt{1+\epsilon^2\rho'^2}dx_1+
\int_{-\frac{L}{2}}^{\frac{L}{2}}u_k^{\epsilon}(x_1,-\epsilon\rho(x_1)) \phi\sqrt{1+\epsilon^2\rho'^2}dx_1 \big )\\
&=\frac{\sigma_k^{\epsilon}}{\epsilon}\big ( \int_{-\frac{L}{2}}^{\frac{L}{2}}v_k^{\epsilon}(x_1,\rho(x_1)) \phi\sqrt{1+\epsilon^2\rho'^2}dx_1+ \int_{-\frac{L}{2}}^{\frac{L}{2}}v_k^{\epsilon}(x_1,-\rho(x_1)) \phi\sqrt{1+\epsilon^2\rho'^2}dx_1 \big )
\end{align*}
We know that $v_k^{\epsilon} \rightharpoonup \overline{V}_k$  in $H^1(T_1)$ and $\overline{V}_k$ depends only on the variable $x_1$, we introduce the function $V_k$ that is the restriction of $\overline{V}_k$ to the variable $x_1$. We let $\epsilon$ goes to $0$ and  we obtain
\begin{equation*}
\int_{T_1} \frac{\partial V_k }{\partial y_1} \frac{\partial \phi }{\partial y_1} dy=2\beta \int_{-\frac{L}{2}}^{\frac{L}{2}} V_k \phi dx_1.
\end{equation*}
Integrating by parts in the left hand side, we finally obtain
\begin{equation*}
-2\int_{-\frac{L}{2}}^{\frac{L}{2}} \frac{d}{dx_1} \big (  \rho \frac{d}{dx_1}  V_k \big ) \phi dx_1=2\beta \int_{-\frac{L}{2}}^{\frac{L}{2}} V_k \phi dx_1.
\end{equation*}
This relation is true for every test function $\phi\in C_c^{\infty}(-\frac{L}{2},\frac{L}{2})$ so we have that $V_k$ and $\beta$ must have to satisfy the following differential equation
\begin{equation}\label{ep11}
-\frac{d}{dx}\big(\rho(x)\frac{d V_k}{dx}(x)\big)=\beta V_k(x)  \qquad  x\in \big(-\frac{L}{2},\frac{L}{2} \big).
\end{equation}

We  find now the boundary conditions associated to this equation. Fix a real number $\xi>0$ and  define the extended function $\overline{\rho}$ in the following way 
\begin{equation*}
\overline{\rho}=\begin{cases}
\vspace{0.2cm}
  \rho(-\frac{L}{2}) \qquad \text{if}\,\, -\frac{L}{2}-\xi\leq x_1\leq -\frac{L}{2} \\
   \vspace{0.2cm}
  \rho(x_1) \qquad \,\,\, \text{if}\,\, -\frac{L}{2}\leq x_1\leq \frac{L}{2} \\
   \rho(\frac{L}{2}) \qquad \text{if}\,\, \frac{L}{2}\leq x_1\leq \frac{L}{2}+\xi.
\end{cases}
\end{equation*}
We define the extended tube $E_{\epsilon}$
\begin{equation}
E_{\epsilon}=\big \{x=(x_1,x_2)\in \mathbb{R}^2 | -\frac{L}{2}-\xi\leq x_1\leq \frac{L}{2}+\xi, |x_2|< \epsilon\overline{\rho}(x_1)\big \},
\end{equation}
and we choose $\xi$ in such a way that $E_{\epsilon}\subset \Omega_{\epsilon}$. Now, repeating all the arguments in Lemma \ref{lemc2}, we obtain that  
\begin{equation*}
v_k^{\epsilon} \rightharpoonup \overline{V}_k \quad \text{in} \quad H^1(E_1)
\end{equation*}
and $\overline{V}_k$ depends only on $x_1$. We also know from Lemma \ref{lemc1} that $u_k^{\epsilon}$ locally uniformly converge to $c_{1,k}$ in $D_1$. From this fact we have that 
\begin{equation*}
v_k^{\epsilon}(-\frac{L}{2}-\delta,0)\rightarrow c_{1,k}=V_k(-\frac{L}{2}-\delta)\quad \forall \, \xi\geq \delta>0,
\end{equation*}
where $V_k$ is the restriction of $\overline{V}_k$ to the variable $x_1$.
We know that $\overline{V}_k\in H^1(E_1)$, from embedding theorem $V_k$ is continuous so we finally obtain 
\begin{equation}\label{cbn2}
V_k(-\frac{L}{2})=c_{1,k}=\lim_{\delta\rightarrow 0} V_k(-\frac{L}{2}-\delta).
\end{equation} 
Similarly,
\begin{equation*}
V_k(\frac{L}{2})=c_{2,k}.
\end{equation*} 

We use the variational formulation of the Steklov eigenvalue with a test function $\psi$ defined on all $\Omega_{\epsilon}$ and that depends only on $x_1$,
\begin{equation*}
\int_{\Omega_{\epsilon}} \frac{\partial u_k^{\epsilon} }{\partial x_1} \frac{\partial \psi }{\partial x_1} dx=\sigma_k^{\epsilon} \int_{\partial \Omega_{\epsilon}}  u_k^{\epsilon} \psi ds.
\end{equation*}

We repeat all the computations above and letting $\epsilon$ goes to $0$, we obtain 
\begin{equation*}
2\int_{-\frac{L}{2}}^{\frac{L}{2}}\rho \frac{d V_k }{dx_1} \frac{d \psi }{d x_1} dx_1=\beta\Big (  V_k \big (-\frac{L}{2}\big )\int_{\partial D_1}\psi ds +   V_k \big ( \frac{L}{2}\big )\int_{\partial D_2} \psi ds+ 2\int_{-\frac{L}{2}}^{\frac{L}{2}} V_k \psi dx_1 \Big ).
\end{equation*}

Integrating by parts the left hand side and recalling the equation \eqref{ep11} we finally get
\begin{equation*}
\rho(\frac{L}{2})\frac{dV_k}{dx}(\frac{L}{2})\psi(\frac{L}{2})-\rho(-\frac{L}{2})\frac{dV_k}{dx}(-\frac{L}{2})\psi(-\frac{L}{2})=\beta\big (  V_k \big (-\frac{L}{2}\big )\int_{\partial D_1}\psi ds +   V_k \big ( \frac{L}{2}\big )\int_{\partial D_2} \psi ds \big ).
\end{equation*}
Choosing a test function such that $\psi=1$ in $D_1$ and $\psi=0$ in $D_2$, we get the first boundary condition
\begin{equation*}
 \rho(-\frac{L}{2})\frac{dV_k}{dx}(-\frac{L}{2})=-\frac{\beta}{2}P(D_1)V_k(-\frac{L}{2}) 
\end{equation*} 
 and, similarly, chosing a test function such that $\psi=0$ in $D_1$ and $\psi=1$ in $D_2$ we get the second boundary condition 
\begin{equation*}
 \rho(\frac{L}{2})\frac{dV_k}{dx}(\frac{L}{2})=\frac{\beta}{2}P(D_2)V_k(\frac{L}{2}).
\end{equation*}

We finally obtain the following eigenvalue problem for $\beta$
\begin{equation}\label{ST1}
\begin{cases}
\vspace{0.2cm}
   -\frac{d}{dx}\big(\rho(x)\frac{d V_k}{dx}(x)\big)=\beta V_k(x)  \qquad  x\in \big(-\frac{L}{2},\frac{L}{2} \big) \\
\vspace{0.2cm}
      \rho(-\frac{L}{2})\frac{dV_k}{dx}(-\frac{L}{2})=-\frac{\beta}{2}P(D_1)V_k(-\frac{L}{2}) \\
      \rho(\frac{L}{2})\frac{dV_k}{dx}(\frac{L}{2})=\frac{\beta}{2}P(D_2)V_k(\frac{L}{2}).
\end{cases}
\end{equation}

To be able to conclude, it remains to prove that $V_k$ is not the zero function. If $V_k$ would be zero,
from the normalization $\int_{\partial\Omega_\epsilon} {u _k^\epsilon}^2=1$ and the convergence on the extended tube, we would have
$$1=P(D_1) c_{1,k}^2 + P(D_2) c_{2,k}^2 + 2\int_{L/2}^{L/2} V_k^2.$$
Therefore, $c_{1,k}$ or $c_{2,k}$ would not be zero yielding a contradiction since the function
$V_k$ being in $H^1(-\frac{L}{2}-\delta,\frac{L}{2}+\delta)$ is continuous and thus cannot be constant (different from zero) on
$(-\frac{L}{2}-\delta, -\frac{L}{2})$ and zero after $-\frac{L}{2}$.
Therefore we have proved that there exist $j\in\mathbb{N}$ such that $0\leq j\leq k$ and $\beta=\mu_j$, where $\mu_j$ is the $j-$th eigenvalue of the problem \eqref{ep1}. 
\end{proof}

\medskip
\noindent{\bf Step 2.} We have just proved that $\sigma_k^{\epsilon}\sim \mu_j \epsilon$, with $0\leq j\leq k$. In this step we justify that $j=k$ .
We  denote by $\widetilde{V}_k$ the function constructed by taking the  $k-$th eigenfunction of  problem \eqref{ep1} and extending it constantly in the $x_2$ variable and equally constant to $V_k(-\frac{L}{2})$ in $D_1$ and to $V_k(\frac{L}{2})$ in $D_2$.

We proceed by induction on the eigenvalue rank $k$. The case $k=0$ is obvious. Now suppose that for all $0\leq j\leq k-1$ we have that $v_j^{\epsilon} \rightharpoonup \overline{V}_j$ in $H^1(T_1)$ and $\sigma_j^{\epsilon}\sim \mu_j \epsilon$.

Now we will prove that $ \mu_k\epsilon+o(\epsilon)\leq \sigma_k^{\epsilon}$. By contradiction we suppose that there exists  $j\in\mathbb{N}$ such that $0\leq j\leq k-1$, $v_k^{\epsilon} \rightharpoonup \overline{V}_j$ in $H^1(T_1)$ and $\sigma_k^{\epsilon}\sim \mu_j \epsilon$. From the orthogonality of the Steklov eigenfunctions we have the following equality 
\begin{equation*}
0=\lim_{\epsilon\rightarrow 0} \int_{\partial \Omega_{\epsilon}}u_k^{\epsilon}  \widetilde{V}_j ds+\int_{\partial \Omega_{\epsilon}}u_k^{\epsilon}(u_j^{\epsilon}- \widetilde{V}_j)ds.
\end{equation*}

For the first term, from \eqref{denn1},  we have that $\lim_{\epsilon\rightarrow 0} \int_{\partial \Omega_{\epsilon}}u_k^{\epsilon}\widetilde{V}_j=2$. For the second term, we recall the inductive hypothesis $v_j^{\epsilon} \rightharpoonup \overline{V}_j$, using the same argument in the proof of Lemma \ref{lep1} we conclude also the equality \eqref{cbn2} and, using Cauchy-Schwartz inequality, we obtain 
\begin{equation*}
|\lim_{\epsilon\rightarrow 0}\int_{\partial \Omega_{\epsilon}}u_k^{\epsilon}(u_j^{\epsilon}- \overline{V}_j)ds|\leq \lim_{\epsilon\rightarrow 0} ||u_k^{\epsilon}||_{L^2(\partial \Omega_{\epsilon})}||u_j^{\epsilon}-\overline{V}_j||_{L^2(\partial \Omega_{\epsilon})}=0.
\end{equation*} 
This is a contraddiction, we conclude that $ \mu_k\epsilon+o(\epsilon)\leq \sigma_k^{\epsilon}$. Now recalling that $\sigma_k^{\epsilon}\leq \mu_k\epsilon+o(\epsilon)$ (see inequality \eqref{stekin1}) we can conclude that 
\begin{equation*}
\sigma_k^{\epsilon}\sim \mu_k \epsilon+o(\epsilon).
\end{equation*}
We also conclude that
\begin{equation*}
u_k^{\epsilon}(x_1,\epsilon x_2) \rightharpoonup \overline{V}_k(x_1,x_2) \quad \text{in} \quad H^1(T_1),
\end{equation*}
where $\overline{V}_k$ is the  $k-$th eigenfunction of the problem \eqref{ep1} constantly extended to $x_2$.  We end the proof by proving that the convergence is true not only up to a subsequence but is true for all the sequence. We have seen that the only possible accumulation point is $\overline{V}_k$
eigenfunction of Problem \eqref{ST1}. 
Now it is a classical result for Sturm-Liouville type problem that any eigenfunction
is simple: use the ODE to prove that the Wronskian is constant and the boundary conditions
to prove that it is zero, yielding the result. 


From the uniqueness of the accumulation point we conclude that the convergence holds for the whole
sequence. This concludes the proof of Theorem \ref{th1}

\section{The case $n\geq 3$ and $k\geq2$. Proof of Theorem \ref{th2}.}
In this section we will prove the second part of Theorem \ref{th2}. We will use the following notation, take $x\in \mathbb{R}^n$ then we write $x=(x_1,x')$ where $x_1\in \mathbb{R}$ and $x'\in \mathbb{R}^{n-1} $
\subsection{Upper bound for the Steklov eigenvalue}
In this section we prove un upper bound for all the Steklov eigenvalues. In the following lemma we give an estimate from above of the speed of convergence to zero of the Steklov eigenvalues. 
\begin{lemma}\label{lemb1}
Let $n\geq 3$ and let $\Omega_{\epsilon}\subset \mathbb{R}^n$ be a dumbbell shape domain then
\begin{enumerate}
\item[$\bullet$] for the first Steklov eigenvalue
\begin{equation}\label{ubs1}
\sigma_1^{\epsilon}\leq \sigma_1\epsilon^{n-1}+o(\epsilon^{n-1})
\end{equation}
where $\sigma_1>0$ is the unique positive number such that the following differential equation has a non-trivial solution:
\begin{equation*}
\begin{cases}
\vspace{0.2cm}
   -w_{n-1}\frac{d}{dx}\big(\rho^{n-1}(x)\frac{d V_1}{dx}(x)\big)=0  \qquad  x\in \big(-\frac{L}{2},\frac{L}{2} \big) \\
   \vspace{0.2cm}
   \rho^{n-1}(-\frac{L}{2})\frac{dV_1}{dx}(-\frac{L}{2})=-\frac{\sigma_1}{\omega_{n-1}}P(D_1)V_1(-\frac{L}{2}) \\
   \rho^{n-1}(\frac{L}{2})\frac{dV_1}{dx}(\frac{L}{2})=\frac{\sigma_1}{\omega_{n-1}}P(D_2)V_1(\frac{L}{2}).
\end{cases}
\end{equation*}
\item[$\bullet$] For all the other Steklov eigenvalues ($k\geq 2$)
\begin{equation}\label{ubs2}
\sigma_k^{\epsilon}\leq \lambda^{\epsilon}_k \epsilon+o(\epsilon) 
\end{equation}
where $\lambda^{\epsilon}_k$ is defined by the following $1-$dimensional eigenvalue problem:
\begin{equation}\label{ep4}
\begin{cases}
\vspace{0.2cm}
   -\omega_{n-1}\frac{d}{dx}\big(\rho^{n-1}(x)\frac{d V_k}{dx}(x)\big)=\lambda^{\epsilon}_k\omega_{n-2}\rho^{n-2}(x)V_k(x)  \qquad  x\in \big(-\frac{L}{2},\frac{L}{2} \big) \\
   \vspace{0.2cm}
  \rho^{n-1}(-\frac{L}{2})\frac{dV_k}{dx}(-\frac{L}{2})=-\frac{\lambda^{\epsilon}_k}{\epsilon^{n-2}}P(D_1)V_k(-\frac{L}{2})\\
  \rho^{n-1}(\frac{L}{2})\frac{dV_k}{dx}(\frac{L}{2})=\frac{\lambda^{\epsilon}_k}{\epsilon^{n-2}}P(D_2)V_k(\frac{L}{2}).
\end{cases}
\end{equation}
\end{enumerate}
\end{lemma}

\begin{proof}
We introduce the following functional space
\begin{equation}
H_{co}^1(\Omega_{\epsilon})=\big \{u\in H^1(\Omega_{\epsilon}) | u\equiv c_i\; \text{in}\;  D_i \;
\int_{\partial\Omega_\epsilon} u =0,\; \text{and}\;u\; \text{depends only on}\; x_1\; \text{in}\;  T_{\epsilon} \big \}
\end{equation}
and denote $\sigma_k^{co}(\Omega_{\epsilon})$ the (pseudo) $k$-th Steklov eigenvalue computed by  replacing the Sobolev space $H^1(\Omega _\epsilon)$ with $H_{co}^1(\Omega_{\epsilon})$ in the 
variational formulation using the Rayleigh quotient.
Since $H_{co}^1(\Omega_{\epsilon})$ is a subspace of $H^1(\Omega_{\epsilon})$, we obtain: 
\begin{align*}
\sigma_1^{\epsilon}\leq \sigma_1^{co}(\Omega_{\epsilon})&=\inf_{0\neq u\in H_{co}^1(\Omega_{\epsilon})}\frac{\int_{\Omega_{\epsilon}}|\nabla u|^2dx}{\int_{\partial\Omega_{\epsilon}}u^2d\mathcal{H}^{n-1}}\\
&\leq \inf_{0\neq u\in H_{co}^1(\Omega_{\epsilon})}\frac{\epsilon^{n-1}w_{n-1}\int_{-\frac{L}{2}}^{\frac{L}{2}}\rho^{n-1}(u')^2dx_1}{P(D_1)u^2(-\frac{L}{2})+P(D_2)u^2(\frac{L}{2})}+o(\epsilon^{n-1})\\
&\leq\sigma_1 \epsilon^{n-1}+o(\epsilon^{n-1}).
\end{align*}
The last inequality is true because the quantity
\begin{equation*}
\inf_{0\neq u\in H_{co}^1(\Omega_{\epsilon})}\frac{w_{n-1}\int_{-\frac{L}{2}}^{\frac{L}{2}}\rho^{n-1}(u')^2dx_1}{P(D_1)u^2(-\frac{L}{2})+P(D_2)u^2(\frac{L}{2})}
\end{equation*}
is equal to $\sigma_1$ that is the unique positive number such that the following differential equation has a non-trivial solution:
\begin{equation*}
\begin{cases}
\vspace{0.2cm}
   -w_{n-1}\frac{d}{dx}\big(\rho^{n-1}(x)\frac{d V_1}{dx}(x)\big)=0  \qquad  x\in \big(-\frac{L}{2},\frac{L}{2} \big) \\
   \vspace{0.2cm}
   \rho^{n-1}(-\frac{L}{2})\frac{dV_1}{dx}(-\frac{L}{2})=-\frac{\sigma_1}{w_{n-1}}P(D_1)V_1(-\frac{L}{2}) \\
   \rho^{n-1}(\frac{L}{2})\frac{dV_1}{dx}(\frac{L}{2})=\frac{\sigma_1}{w_{n-1}}P(D_2)V_1(\frac{L}{2}).
\end{cases}
\end{equation*}
We now prove the second part of the lemma. 
We start by noticing that from the geometric properties of $\partial T_{\epsilon}^e$ we can compute the surface measure and we obtain that
\begin{equation}\label{surfmeas2}
d\mathcal{H}^{n-1}\llcorner \partial T_{\epsilon}^e =\epsilon^{n-2}\rho^{n-2}\sqrt{1+\epsilon^2\rho'^2}dx_1d\varphi_1...d\varphi_{n-2}
\end{equation}
Let $S_{k}$ be the family of all the $k-$dimensional subspaces of the functional space $H_{co}^1(\Omega_{\epsilon})$ with $k\geq 2$, as above we have the following inequalities 
$$
\sigma_k^{\epsilon}\leq \sigma_k^{co}(\Omega_{\epsilon})=\inf_{E\in S_{k+1}}\sup_{u\in E}\frac{\int_{\Omega_{\epsilon}}|\nabla u|^2dx}{\int_{\partial\Omega_{\epsilon}}u^2d\mathcal{H}^{n-1}}$$
$$\leq \inf_{E\in S_{k+1}}\sup_{u\in E}\frac{\epsilon^{n-1}w_{n-1}\int_{-\frac{L}{2}}^{\frac{L}{2}}\rho^{n-1}(u')^2dx_1}{P(D_1)u^2(-\frac{L}{2})+P(D_2)u^2(\frac{L}{2})+\epsilon^{n-2}w_{n-2}\int_{-\frac{L}{2}}^{\frac{L}{2}}u^2\rho^{n-2}\sqrt{1+\epsilon^2\rho'^2}dx_1}+o(\epsilon)$$
$$\leq \epsilon \inf_{E\in S_{k+1}}\sup_{u\in E}\frac{w_{n-1}\int_{-\frac{L}{2}}^{\frac{L}{2}}\rho^{n-1}(u')^2dx_1}{\frac{P(D_1)}{\epsilon^{n-2}}u^2(-\frac{L}{2})+\frac{P(D_2)}{\epsilon^{n-2}}u^2(\frac{L}{2})+w_{n-2}\int_{-\frac{L}{2}}^{\frac{L}{2}}u^2\rho^{n-2}dx_1}+o(\epsilon)$$
$$\leq \lambda^{\epsilon}_k\epsilon+o(\epsilon).$$
Where the last inequality is true because the quantity
\begin{equation*}
\frac{w_{n-1}\int_{-\frac{L}{2}}^{\frac{L}{2}}\rho^{n-1}(u')^2dx_1}{\frac{P(D_1)}{\epsilon^{n-2}}u^2(-\frac{L}{2})+\frac{P(D_2)}{\epsilon^{n-2}}u^2(\frac{L}{2})+w_{n-2}\int_{-\frac{L}{2}}^{\frac{L}{2}}u^2\rho^{n-2}dx_1}
\end{equation*}
is the Rayleigh quotient of the following eigenvalue problem that depends on $\epsilon$

\begin{equation}
\begin{cases}
\vspace{0.2cm}
   -w_{n-1}\frac{d}{dx}\big(\rho^{n-1}(x)\frac{d V_k}{dx}(x)\big)=\lambda^{\epsilon}_kw_{n-2}\rho^{n-2}(x)V_k(x)  \qquad  x\in \big(-\frac{L}{2},\frac{L}{2} \big) \\
   \vspace{0.2cm}
  \rho^{n-1}(-\frac{L}{2})\frac{dV_k}{dx}(-\frac{L}{2})=-\frac{\lambda^{\epsilon}_k}{\epsilon^{n-2}}P(D_1)V_k(-\frac{L}{2})\\
  \rho^{n-1}(\frac{L}{2})\frac{dV_k}{dx}(\frac{L}{2})=\frac{\lambda^{\epsilon}_k}{\epsilon^{n-2}}P(D_2)V_k(\frac{L}{2}).
\end{cases}
\end{equation}

\end{proof}

\subsection{Convergence of eigenfunctions}
We begin with the convergence on the two regions $D_i$ where $i=1,2$. The proof of the following lemma is the same of the proof of Lemma \ref{lemc1}, so we do not repeat it.
\begin{lemma}\label{lemcn1}
Let $k\geq 1$ we have (up to a sub-sequence that we still denote by $u_k^{\epsilon}$) 
\begin{align*}
u_k^{\epsilon} &\rightharpoonup c_{i ,k}\quad \text{in} \quad H^1(D_i),\\
u_k^{\epsilon} &\rightarrow c_{i,k} \quad \text{locally uniformly in} \quad D_i.
\end{align*} 
where $c_{i,k}$ are constants
\end{lemma}
We study the behaviour of the eigenfunctions in the tube $T_{\epsilon}$. For every $k\geq 2$ we define the following rescaled functions
\begin{equation*}
v_k^{\epsilon}(x_1,x')=\epsilon^{\frac{n-2}{2}}u_k^{\epsilon}(x_1,\epsilon x')\quad \forall \, (x_1,x')\in T_1
\end{equation*}
\begin{lemma}\label{lemcn2}Let $n\geq 3$ and $k\geq 2$. There exists $\overline{V}_k\in H^1(T_1)$   which depends only on the variable $x_1$ such that
\begin{equation*}
v_k^{\epsilon} \rightharpoonup \overline{V}_k \quad \text{in} \quad H^1(T_1),
\end{equation*}
 up to a sub-sequence (that we still denote by $v_k^{\epsilon}$).
\end{lemma}
\begin{proof} Below, by $C$ we denote a constant which may change from line to line. We start with the bound of $||\nabla v_k^{\epsilon}||_{L^2(T_1)}$ 
\begin{equation*}
\int_{T_1}|\nabla v_k^{\epsilon}|^2 dx\leq \int_{T_1} \Big( \frac{\partial v_k^{\epsilon}}{\partial x_1} \Big )^2+\frac{1}{\epsilon^2} |\nabla_{x'} v_k^{\epsilon} |^2 dx=\frac{1}{\epsilon}\int_{T_{\epsilon}}|\nabla u_k^{\epsilon}|^2 dy\leq C
\end{equation*}
where we performed the change of coordinates $y_1=x_1$, $y'=\epsilon x'$. The last inequality is true because of \eqref{ubs2}.  We want now to bound $||v_k^{\epsilon}||_{L^2(T_1)}$.  By Fubini Tonelli we have:
\begin{equation}\label{bv1}
\int_{T_1}(v_k^{\epsilon})^2dx=\frac{1}{\epsilon}\int_{-\frac{L}{2}}^{\frac{L}{2}}\int_{B_{\epsilon\rho(x_1)}^{n-1}(x_1)}(u_k^{\epsilon}(x_1,x'))^2d\mathcal{H}^{n-1}dx_1,
\end{equation}
where $B_{\epsilon\rho(x_1)}^{n-1}(x_1)$ is the $n-1$ dimensional ball centered in $x_1$ with radius $\epsilon\rho(x_1)$.
Using the characterization of Robin eigenvalues, we obtain for all $x_1\in(-\frac{L}{2},\frac{L}{2})$:
\begin{equation}\label{pf1}
\int_{B_{\epsilon\rho(x_1)}^{n-1}(x_1)}(u_k^{\epsilon})^2d\mathcal{H}^{n-1}\leq \frac{\int_{B_{\epsilon\rho(x_1)}^{n-1}(x_1)}|\nabla u_k^{\epsilon}|^2d\mathcal{H}^{n-1}+\int_{\partial B_{\epsilon\rho(x_1)}^{n-1}(x_1)}(u_k^{\epsilon})^2d\mathcal{H}^{n-2}}{\lambda_1(B_{\epsilon\rho(x_1)}^{n-1}(x_1),1)}
\end{equation}
where $\lambda_1(B_{\epsilon\rho(x_1)}^{n-1}(x_1),1)$ is the first Robin eigenvalue with parameter $1$ of the ball $B_{\epsilon\rho(x_1)}^{n-1}(x_1)$.

Now we recall that (see \cite{GS05})
\begin{equation*}
\lambda_1(B_{\epsilon},1)\sim  \frac{n}{\epsilon}
\end{equation*}
where $B_{\epsilon}$ is the $n-$dimensional ball of radius $\epsilon$. In particular we have  
\begin{equation*}
\lambda_1(B_{\epsilon\rho(x_1)}^{n-1}(x_1),1)\sim  \frac{n-1}{\epsilon \rho(x_1)}.
\end{equation*}
This asymptotic formulae together with  \eqref{pf1} and  \eqref{bv1} give
\begin{align*}
\int_{T_1}(v_k^{\epsilon})^2dx\leq C \big (\int_{T_{\epsilon}}(u_k^{\epsilon})^2dx+\int_{\partial T_{\epsilon}^e}(u_k^{\epsilon})^2ds \big )\leq C,
\end{align*}
where the last inequality is true because  $||\nabla u_k^{\epsilon}||_{L^2(D_1)}\leq C\epsilon$ and  $||u_k^{\epsilon}||_{L^2(\partial\Omega_{\epsilon})}=1$.

We conclude that there exists $\overline{V}_k\in H^1(T_1)$ such that (up to a sub-sequence that we still denote by $v_k^{\epsilon}$) 
\begin{equation*}
v_k^{\epsilon} \rightharpoonup \overline{V}_k \quad \text{in} \quad H^1(T_1).
\end{equation*}

We finish the proof by showing that $\overline{V}_k$ does not depend on $x_i$ for all $i\geq 2$, indeed 
\begin{equation*}
\int_{T_1} \big ( \frac{\partial v_k^{\epsilon} }{\partial x_i} \big )^2dx=\epsilon \int_{T_{\epsilon}} \big ( \frac{\partial u_k^{\epsilon} }{\partial x_i} \big )^2dx\leq C \epsilon^2 \rightarrow 0.
\end{equation*} 
\end{proof}
Let now $V_k$ be the restriction of $\overline{V}_k$ to the variable $x_1$,
the main goal now is to prove that $V_k$ is not the zero function and $V_k(-\frac{L}{2})=V_k(-\frac{L}{2})=0$. In order to reach this result we start by some consideration about the constants $c_{i,k}$ in Lemma \ref{lemcn1}.

We know that $\int_{\partial\Omega_{\epsilon}}(u^{\epsilon}_k)^2d\mathcal{H}^{n-1}=1$ for all $\epsilon$ and it is easy to see that, by change of variable, we have that $\int_{\partial T_{\epsilon}}(u^{\epsilon}_k)^2d\mathcal{H}^{n-1}=\int_{\partial T_1}(v^{\epsilon}_k)^2d\mathcal{H}^{n-1}$. Now by Lemmas \ref{lemcn1} and \ref{lemcn2} we obtain that for all $k\geq 2$,
\begin{equation*}
\int_{\partial \Omega_{\epsilon}}(u^{\epsilon}_k)^2d\mathcal{H}^{n-1}\rightarrow c_{1,k}^2P(D_1)+c_{2,k}^2P(D_2)+w_{n-2}\int_{-\frac{L}{2}}^{\frac{L}{2}}V_k^2(x_1)\rho^{n-2}(x_1)dx_1,
\end{equation*}
so if we prove that $c_{1,k}=c_{2,k}=0$, we conclude that  $V_k$ is not identically zero. We now prove that $c_{1,k}=c_{2,k}=0$ for $k\geq 2$.

We note that for all $k\geq 1$, by Cauchy-Schwarz inequality 
\begin{equation*}
|\int_{\partial T_{\epsilon}}u^{\epsilon}_kd\mathcal{H}^{n-1}|\leq P(T_{\epsilon})^{\frac{1}{2}}||u^{\epsilon}_k||_{L^2(\partial T_{\epsilon})}\rightarrow 0,
\end{equation*}
and we also know that  $\int_{\partial\Omega_{\epsilon}}u^{\epsilon}_k d\mathcal{H}^{n-1}=0$, from Lemmas \ref{lemcn1} and \ref{lemcn2} we obtain that for all $k\geq 1$
\begin{equation}\label{ec1}
 c_{1,k}P(D_1)+c_{2,k}P(D_2)=0.
\end{equation}
Like in the proof of Lemma \ref{lep1}, we introduce the extended tube $E_{\epsilon}$. We fix a real number $\xi>0$ and we define the extended function $\overline{\rho}$ in the following way 
\begin{equation*}
\overline{\rho}=\begin{cases}
\vspace{0.2cm}
  \rho(-\frac{L}{2}) \qquad \text{if}\,\, -\frac{L}{2}-\xi\leq x_1\leq -\frac{L}{2} \\
   \vspace{0.2cm}
  \rho(x_1) \qquad \,\,\, \text{if}\,\, -\frac{L}{2}\leq x_1\leq \frac{L}{2} \\
   \rho(\frac{L}{2}) \qquad \text{if}\,\, \frac{L}{2}\leq x_1\leq \frac{L}{2}+\xi.
\end{cases}
\end{equation*}
We define the extended tube $E_{\epsilon}$
\begin{equation}\label{extt}
E_{\epsilon}=\big \{x=(x_1,x')\in \mathbb{R}^n | -\frac{L}{2}-\xi\leq x_1\leq \frac{L}{2}+\xi, |x'|< \epsilon\overline{\rho}(x_1)\big \},
\end{equation}
and we choose $\xi$ in such a way that $E_{\epsilon}\subset \Omega_{\epsilon}$.
\begin{figure}
\centering
\tdplotsetmaincoords{0}{0}
\begin{tikzpicture}[tdplot_main_coords, scale=0.7]
\fill[left color=gray!50!black,right color=gray!50!black,middle color=gray!50,shading=axis,opacity=0.15] (-4.45,0,0) circle (2.47); 
\fill[left color=gray!50!black,right color=gray!50!black,middle color=gray!50,shading=axis,opacity=0.15] (5,0,0) circle (3);
\fill[left color=gray!50!black,right color=gray!50!black,middle color=gray!50,shading=axis,opacity=0.30] (-2.20,1,0) .. controls (-1,0,0,).. (0,0.5,0).. controls (1,1,0,)..  (2.05,0.5,0)  arc (30:-30:1).. controls (1,-1,0,).. (0,-0.5,0) .. controls (-1,0,0,).. (-2.20,-1,0) arc (200:160:2.9);
\draw (2.05,0.5,0)  arc (30:-30:1);
\draw (2.05,0.5,0) arc (170.5:-170.5:3);
\draw [dashed] (2.05,0.5,0) arc (170.5:210:3);
\draw (8,0,0) arc (5:-163:3 and 0.8);
\draw [dashed] (8,0,0) arc (5:197:3 and 0.8);
\draw (-2.20,1,0) .. controls (-1,0,0,).. (0,0.5,0).. controls (1,1,0,)..  (2.05,0.5,0)  arc (30:-30:1).. controls (1,-1,0,).. (0,-0.5,0) .. controls (-1,0,0,).. (-2.20,-1,0) arc (200:160:2.9);
\draw (-2.20,1,0) arc(24:337:2.47);
\draw [dashed ] (-2.20,1,0) arc(24:-24:2.47); 
\draw (-6.92,0,0) arc(180:328:2.47 and 0.8);
\draw [dashed] (-6.92,0,0) arc(180:-32:2.47 and 0.8);

\draw (1,0.88,0,) arc(90:-90:0.4 and 0.87);
\draw [dashed](1,0.88,0,) arc(90:270:0.4 and 0.87);
\draw (-1,0.18,0,) arc(90:-90:0.08 and 0.18);
\draw [dashed] (-1,0.18,0,) arc(90:270:0.08 and 0.18);
\fill[left color=gray!50!black,right color=gray!50!black,middle color=gray!50,shading=axis,opacity=0.30] (-2.20,-1,0) arc (200:160:2.9) -- (-4.5,1,0) arc (90:270:0.2 and 1); 
\draw [dashed] (-2.20,1,0) -- (-4.5,1,0) arc (90:270:0.2 and 1) -- (-2.20,-1,0); 
\draw [dashed] (-4.5,1,0) arc (90:-90:0.2 and 1);
\fill[left color=gray!50!black,right color=gray!50!black,middle color=gray!50,shading=axis,opacity=0.30] (2.05,0.5,0) -- (4,0.5,0) arc (90:-90:0.2 and 0.53) -- (2.05,-0.5,0);
\draw [dashed] (2.05,0.5,0) -- (4,0.5,0) arc (90:-90:0.2 and 0.53) -- (2.05,-0.5,0);
\draw [dashed] (4,0.5,0) arc (90:270:0.2 and 0.53);
\draw (0,0,0) node {$E_{\epsilon}$};
\end{tikzpicture}
\caption{Extended tube $E_{\epsilon}$.}
\label{fig3}
\end{figure}
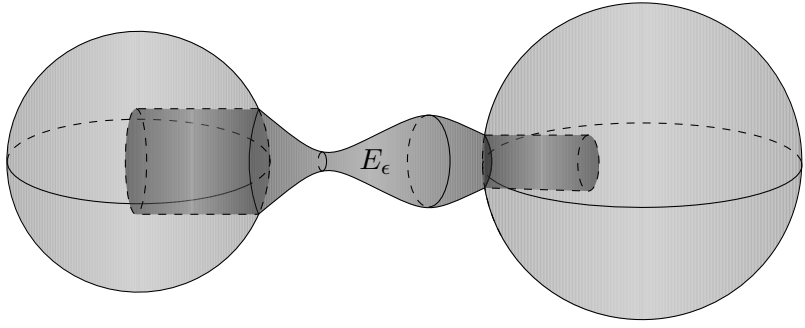

Repeating all the arguments in Lemma \ref{lemcn2}, we obtain that for all $k\geq 2$ 
\begin{equation*}
v_k^{\epsilon} \rightharpoonup \overline{V}_k \quad \text{in} \quad H^1(E_1)
\end{equation*}
and $\overline{V}_k$ depends only on $x_1$. We also know from Lemma \ref{lemcn1} that, for all $k\geq 1$, $u_k^{\epsilon}$ locally uniformly converge to $c_{1,k}$ in $D_1$ and to $c_{2,k}$ in $D_2$.

The following lemma contains a key result on the asymptotic behaviour of the first eigenfunctions.
\begin{lemma}\label{lemcn3}
Let $n\geq 3$ and let $\Omega_{\epsilon}\subset \mathbb{R}^n$ be a dumbbell shape domain then the following holds 
\begin{equation}
\lim_{\epsilon\rightarrow 0}\int_{\partial T_{\epsilon}}(u^{\epsilon}_1)^2d\mathcal{H}^{n-1}=0.
\end{equation} 
In particular $c_{1,1}\neq 0, c_{2,1}\neq 0$.
\end{lemma}
\begin{proof}
By contradiction we assume that there exists $\alpha>0$ such that for $\epsilon$ small enough
\begin{equation}\label{contr1}
\int_{\partial T_{\epsilon}}(u^{\epsilon}_1)^2d\mathcal{H}^{n-1}\geq \alpha >0.
\end{equation}
 For every $x_1\in(-\frac{L}{2},\frac{L}{2})$ we consider the ball $B_{\epsilon\rho(x_1)}^{n-1}(x_1)$, that is  $x_1-$section of the tube $T_{\epsilon}^e$. We use the Trace Theorem in any section of the tube $T_{1}^e$ and we rescale to the sections of the tube $T_{\epsilon}^e$, we obtain, for all $x_1\in(-\frac{L}{2},\frac{L}{2})$:
\begin{equation*}
\int_{\partial B_{\epsilon\rho(x_1)}^{n-1}(x_1)}(u^{\epsilon}_1)^2d\mathcal{H}^{n-2}\leq \frac{1}{\epsilon} \int_{B_{\epsilon\rho(x_1)}^{n-1}(x_1)}(u^{\epsilon}_1)^2d\mathcal{H}^{n-1}+\epsilon  \int_{B_{\epsilon\rho(x_1)}^{n-1}(x_1)}|\nabla_{x'}u^{\epsilon}_1|^2d\mathcal{H}^{n-1}.
\end{equation*}
Integrating this inequality in $x_1$ we obtain 

\begin{equation*}
\epsilon\int_{\partial T_{\epsilon}^e}(u^{\epsilon}_1)^2d\mathcal{H}^{n-1}\leq\int_{ T_{\epsilon}^e}(u^{\epsilon}_1)^2dx+\epsilon^2  \int_{T_{\epsilon}^e}|\nabla_{x'}u^{\epsilon}_1|^2dx,
\end{equation*}
from this inequality, the fact that $\sigma_1^{\epsilon}\leq C\epsilon^{n-1}$ and inequality \eqref{contr1} we finally have 
\begin{equation}\label{contr2}
\epsilon\frac{\alpha}{2}\leq \int_{ T_{\epsilon}^e}(u^{\epsilon}_1)^2dx.
\end{equation}
From the fact that $\int_{\partial \Omega_{\epsilon}}u^{\epsilon}_1d\mathcal{H}^{n-1}=0$ for all $\epsilon$ we have two cases: either $c_{1,1}=c_{2,1}=0$ or $c_{1,1}$ and  $c_{2,1}$ have opposite sign, say $c_{1,1}>0>c_{2,1}$.

If $c_{1,1}>0>c_{2,1}$, then at $x'$ fixed the function $x_1\rightarrow u_1^{\epsilon}(x_1,x')$ change sign. In particular there exist $C>0$ such that:
\begin{equation*}
\int_{-\frac{L}{2}-\xi}^{\frac{L}{2}+\xi}\big ( \frac{\partial u_k^{\epsilon} }{\partial x_1} \big )^2(x_1,x')dx_1\geq C\int_{-\frac{L}{2}}^{\frac{L}{2}} ( u_k^{\epsilon})^2(x_1,x')dx_1,
\end{equation*}
by integrating in $x'$ we finally obtain 
\begin{equation*}
\int_{E_{\epsilon}}|\nabla u_k^{\epsilon}|^2dx\geq C\epsilon \frac{\alpha}{2}.
\end{equation*}
which is a contradiction to the fact that  $\sigma_1^{\epsilon}\leq C\epsilon^{n-1}$.
 
 If $c_{1,1}=c_{2,1}=0$, we define $g_1^{\epsilon}(x_1,x')=\epsilon^{\frac{n-2}{2}}u_1^{\epsilon}(x_1,\epsilon x')$. Using the same argument in Lemma \ref{lemcn2} we  conclude that there exist $\overline{G}_1\in H^1(E_1)$ such that 
\begin{equation*}
g_1^{\epsilon} \rightharpoonup \overline{G}_1 \quad \text{in} \quad H^1(E_1).
\end{equation*}
From this convergence, the fact that $\int_{\partial\Omega_{\epsilon}}(u^{\epsilon}_k)^2d\mathcal{H}^{n-1}=1$ for all $\epsilon$ and the assumption $c_{1,1}=c_{2,1}=0$ we get
\begin{align*}
\int_{\partial T_1^e} \overline{G}_1^2d\mathcal{H}^{n-1}=1\qquad \mbox{and} \qquad
\int_{\partial D_1\cup \partial D_2}(u_1^{\epsilon})^2d\mathcal{H}^{n-1}\rightarrow 0.
\end{align*}
From this consideration we conclude that 
\begin{equation*}
C\epsilon^{n-1}\geq \sigma_1^{\epsilon}\geq \int_{E_1}|\nabla \overline{G}_1|^2dx,
\end{equation*}
hence $\int_{E_1}|\nabla \overline{G}_1|^2dx=0$. In particular $\overline{G}_1$ is almost everywhere constant $C\neq 0$. Consider a point in the extremity on the extended tube $E_1$, where we know that $u_1^{\epsilon}$ uniformly converge to a constant. More precisely we find $\overline{x}=(\overline{x}_1,\overline{x}')\in E_1$ such that the following holds 
\begin{align*}
g_1^{\epsilon}(\overline{x}_1,\overline{x}')\rightarrow C>0\quad \mbox{and} \quad
g_1^{\epsilon}(\overline{x}_1,\overline{x}')=\epsilon^{\frac{n-2}{2}}u_1^{\epsilon}(\overline{x}_1,\epsilon\overline{x}')\rightarrow 0.
\end{align*} 
This is a contradiction.

\end{proof}
We now prove that  $c_{1,k}=c_{2,k}=0$ for all $k\geq 2$, concluding that  $V_k$ is not identically zero for all $k\geq 2$.
\begin{lemma}\label{lemcn4}
Let $n\geq 3$, $k\geq 2$ and let let $\Omega_{\epsilon}\subset \mathbb{R}^n$ be a dumbbell shape domain. The following holds
\begin{align*}
u_k^{\epsilon} &\rightharpoonup 0\quad \text{in} \quad H^1(D_i),\\
u_k^{\epsilon} &\rightarrow 0 \quad \text{locally uniformly in} \quad D_i.
\end{align*}  
\end{lemma}

\begin{proof} We start by fixing $k\geq 2$.
Form the orthogonality condition of the Steklov eigenfunctions we know that
\begin{equation*}
\int_{\partial \Omega_{\epsilon}}u_1^{\epsilon}u_k^{\epsilon}d\mathcal{H}^{n-1}=\int_{\partial D_1}u_1^{\epsilon}u_k^{\epsilon}d\mathcal{H}^{n-1}+\int_{\partial T_{\epsilon}^e}u_1^{\epsilon}u_k^{\epsilon}d\mathcal{H}^{n-1}+\int_{\partial D_2}u_1^{\epsilon}u_k^{\epsilon}d\mathcal{H}^{n-1}=0.
\end{equation*} 
From this equality and Lemmas \ref{lemcn1}, \ref{lemcn2} and \ref{lemcn3} we  obtain the equality 
\begin{equation*}
c_{1,k}c_{1,1}P(D_1)+c_{2,k}c_{2,1}P(D_2)=0
\end{equation*}
which, together with
\begin{equation*}
\begin{cases}
c_{1,k}P(D_1)+c_{2,k}P(D_2)=0\\
c_{1,1}P(D_1)+c_{2,1}P(D_2)=0,
\end{cases}
\end{equation*}
lead to a system of equations.

We know that $c_{1,1}\neq 0$ and $c_{2,1}\neq 0$  and (without loss of generality) $c_{1,1}>0> c_{2,1}$. Suppose by contradiction that $c_{2,k}\neq 0$, from the system above we have the following 
\begin{equation*}
\frac{c_{1,1}}{c_{2,1}}=\frac{c_{1,k}}{c_{2,k}}=\frac{c_{1,1}c_{1,k}}{c_{2,1}c_{2,k}},
\end{equation*}
it means that $c_{1,1}=c_{2,1}$ and $c_{1,k}= c_{2,k}$ that is a contradiction. We obtain the same conclusion if we suppose that $c_{1,k}\neq 0$.
\end{proof}
Let $V_k$ be the restriction to $x_1$ of the limit eigenfunction $\overline{V}_k$ in Lemma \ref{lemcn2}, in the next lemma we prove that $V_k$ is not constant and we find the boundary conditions of $V_k$.
\begin{lemma}\label{lemcn5}
Let $n\geq 3$, $k\geq 2$ and let $V_k$ be the restriction to $x_1$ of the limit eigenfunction $\overline{V}_k$ in Lemma \ref{lemcn2} . Then
\begin{enumerate}
\item[$\bullet$] $V_k$ is continuous
\item[$\bullet$] $V_k(-\frac{L}{2})=V_k(\frac{L}{2})=0$
\item[$\bullet$] $V_k$ is not constant
\end{enumerate}
\end{lemma} 
\begin{proof}
The first point is immediate because we know that, if we consider the extended tube $E_1$, $V_k\in H^1((-\frac{L}{2}-\xi,\frac{L}{2}+\xi))$, by classical embedding theorem we have $V_k\in C((-\frac{L}{2}-\xi,\frac{L}{2}+\xi))$.
We prove the second point, we know, from Lemma \ref{lemcn4} that $u_k^{\epsilon}$ locally uniformly converge to $0$ in $D_1$. From this fact we have that 
\begin{equation*}
v_k^{\epsilon}(-\frac{L}{2}-\delta,0)=\epsilon^{\frac{n-2}{2}}u_k^{\epsilon}(-\frac{L}{2}-\delta,0)\rightarrow 0=V_k(-\frac{L}{2}-\delta)\quad \forall \, \xi\geq \delta>0.
\end{equation*}
From the continuity of $V_k$ we conclude that 
\begin{equation*}
V_k(-\frac{L}{2})=\lim_{\delta\rightarrow 0} V_k(-\frac{L}{2}-\delta)=0.
\end{equation*}
The same is true for $V_k(\frac{L}{2})$. The tird point is a direct consequence of this, indeed if $V_k$ is constant then $V_k$ must be equal to zero and this is a contradiction with Lemma \ref{lemcn4} 
\end{proof}

\subsection{Proof of Theorem \ref{th2}}
In this section we  prove Theorem \ref{th2}. We will first prove a bound from below for the asymptotic of $\sigma_k^{\epsilon}$. Then we will prove that this bound from below is also a bound from above finding in that way the right asymptotic of $\sigma_k^{\epsilon}$.
\begin{lemma}\label{bb1}
Let $n\geq 3$ and $k\geq 2$  then, for $\epsilon$ small enough 
\begin{equation}\label{ebb1}
\sigma_k^{\epsilon}\geq \alpha_{k-1}\epsilon+o(\epsilon) 
\end{equation}
where $\alpha_k$ is defined by the following $1-$dimensional Dirichlet eigenvalue problem:
\begin{equation*}
\begin{cases}
\vspace{0.2cm}
   -w_{n-1}\frac{d}{dx}\big(\rho^{n-1}(x)\frac{d V_k}{dx}(x)\big)=\alpha_kw_{n-2}\rho^{n-2}(x)V_k(x)  \qquad  x\in \big(-\frac{L}{2},\frac{L}{2} \big) \\
   \vspace{0.2cm}
   V_k(-\frac{L}{2})=0 \\
   V_k(\frac{L}{2})=0.
\end{cases}
\end{equation*}
\end{lemma}
\begin{proof}
We start by showing that there exists a constant $C_k>0$ such that $\sigma_k^{\epsilon}\geq C_k \epsilon$. Indeed we have 
\begin{align*}
\sigma_k^{\epsilon}=\frac{\int_{\Omega_{\epsilon}}|\nabla u_k^{\epsilon}|^2dx}{\int_{\partial\Omega_{\epsilon}}(u_k^{\epsilon})^2d\mathcal{H}^{n-1}}&\geq \frac{\epsilon^{n-2}\int_{T^e_{\epsilon}}|\nabla u_k^{\epsilon}|^2dx}{\epsilon^{n-2}\int_{\partial D_1\cup \partial D_2}(u_k^{\epsilon})^2d\mathcal{H}^{n-1}+\epsilon^{n-2}\int_{\partial T^e_{\epsilon}}(u_k^{\epsilon})^2d\mathcal{H}^{n-1}}\\
&\geq \epsilon \frac{\int_{T^e_{\epsilon}}|\nabla v_k^{\epsilon}|^2dx}{\int_{\partial D_1\cup \partial D_2}(u_k^{\epsilon})^2d\mathcal{H}^{n-1}+\int_{\partial T^e_{\epsilon}}(v_k^{\epsilon})^2d\mathcal{H}^{n-1}}.
\end{align*}
Now from Lemmas \ref{lemcn2} and \ref{lemcn4}, recalling that $V_k$ is the restriction to $x_1$ of the limit eigenfunction $\overline{V}_k$ in Lemma \ref{lemcn2} and recalling also the geometry of the tube, we finally obtain 
\begin{equation}\label{br}
\sigma_k^{\epsilon}=\frac{\int_{\Omega_{\epsilon}}|\nabla u_k^{\epsilon}|^2dx}{\int_{\partial\Omega_{\epsilon}}(u_k^{\epsilon})^2d\mathcal{H}^{n-1}}\geq \epsilon \frac{w_{n-1}\int_{-\frac{L}{2}}^{\frac{L}{2}}(V_k')^2\rho^{n-1}dx_1}{w_{n-2}\int_{-\frac{L}{2}}^{\frac{L}{2}}V_k^2\rho^{n-2}dx_1}.
\end{equation}
From this inequality and Lemma \ref{lemcn5} follow that there exists $C_k>0$ such that $\sigma_k^{\epsilon}\geq C_k \epsilon$. From the variational characterization of the Steklov eigenvalues we have 
\begin{equation*}
\sigma_k^{\epsilon}=\max_{u\in <u_1^{\epsilon},u_2^{\epsilon},...,u_k^{\epsilon}>}\frac{\int_{\Omega_{\epsilon}}|\nabla u|^2dx}{\int_{\partial\Omega_{\epsilon}}u^2d\mathcal{H}^{n-1}},
\end{equation*}  
where $<u_1^{\epsilon},u_2^{\epsilon},...,u_k^{\epsilon}>$ is the subspace of $H^1(\Omega_{\epsilon})$ generated by $u_1^{\epsilon},u_2^{\epsilon},...,u_{k-1}^{\epsilon}$ and $u_k^{\epsilon}$. We know that that the maximum is achieved when $u=u_k^{\epsilon}$, so we have 
\begin{align*}
\sigma_k^{\epsilon}&=\max_{u\in <u_2^{\epsilon},...,u_k^{\epsilon}>}\frac{\int_{\Omega_{\epsilon}}|\nabla u|^2dx}{\int_{\partial\Omega_{\epsilon}}u^2d\mathcal{H}^{n-1}}\\
&\geq \max_{\{\beta_j\}_{j=2}^k}\frac{\sum_{j=2}^k \beta_j^2\int_{T^e_{\epsilon}}|\nabla u_j^{\epsilon}|^2dx}{\sum_{j=2}^k \beta_j^2\int_{\partial\Omega_{\epsilon}}(u_j^{\epsilon})^2d\mathcal{H}^{n-1}}.
\end{align*}
From the inequality above and inequality \eqref{br} we obtain 
\begin{align}\notag
\sigma_k^{\epsilon}&\geq \epsilon \max_{\{\beta_j\}_{j=2}^k}\frac{\sum_{j=2}^k \beta_j^2w_{n-1}\int_{-\frac{L}{2}}^{\frac{L}{2}}(V_j')^2\rho^{n-1}dx_1}{\sum_{j=2}^k \beta_j^2w_{n-2}\int_{-\frac{L}{2}}^{\frac{L}{2}}V_j^2\rho^{n-2}dx_1}\\\label{rq2}
&\geq  \epsilon \max_{V\in <V2,...,V_k>}\frac{w_{n-1}\int_{-\frac{L}{2}}^{\frac{L}{2}}(V')^2\rho^{n-1}dx_1}{w_{n-2}\int_{-\frac{L}{2}}^{\frac{L}{2}}V^2\rho^{n-2}dx_1}.
\end{align}
From Lemma \ref{lemcn5} we know that $V_j\in H_0^1((-\frac{L}{2},\frac{L}{2}))$ for all $2\leq j\leq k$ and from the orthogonality of the Steklov eigenfunctions we also know that dim$<V2,...,V_k>=k-1$. It is easy to check that the ratio \eqref{rq2} is the Rayleigh quotient of the following eigenvalue problem 
\begin{equation*}
\begin{cases}
\vspace{0.2cm}
   -w_{n-1}\frac{d}{dx}\big(\rho^{n-1}(x)\frac{d V_k}{dx}(x)\big)=\alpha_kw_{n-2}\rho^{n-2}(x)V_k(x)  \qquad  x\in \big(-\frac{L}{2},\frac{L}{2} \big) \\
   \vspace{0.2cm}
   V_k(-\frac{L}{2})=0 \\
   V_k(\frac{L}{2})=0.
\end{cases}
\end{equation*}
From \eqref{rq2} we finally conclude that 
\begin{equation*}
\sigma_k^{\epsilon}\geq \alpha_{k-1}\epsilon+o(\epsilon) 
\end{equation*}
\end{proof}
We prove that the bound from below given in \eqref{ebb1} is in fact also a bound from above.
\begin{lemma}\label{lemba1}
Let $\alpha_k$ be the $k-$th eigenvalue of the problem \eqref{ep3} and let $\lambda_k^{\epsilon} $ be $k-$th eigenvalue of the problem \eqref{ep4}, then for $\epsilon$ small enough we have 
\begin{equation}\label{ce1}
\lambda_k^{\epsilon}\leq \alpha_{k-1}+o(\epsilon).
\end{equation} 
\end{lemma}
\begin{proof}
We choose $V_1,..,V_{k-1}$ eigenfunctions of the problem \eqref{ep3}, in particular $V_i(-\frac{L}{2})=V_i(\frac{L}{2})=0$ for all $i=1,...,k-1$. Now we define a function $\psi\in C^{\infty}(-\frac{L}{2},\frac{L}{2})$ such that $\psi(-\frac{L}{2})=1$, $\psi(\frac{L}{2})=0$ and dim$<\psi,V_1,..,V_{k-1}>=k$. Using the variational characterization of the eigenvalue $\lambda_k^{\epsilon}$ we obtain 
\begin{align}\notag
\lambda_k^{\epsilon}&\leq \max_{u\in <\psi,V_1,..,V_{k-1}>}\frac{w_{n-1}\int_{-\frac{L}{2}}^{\frac{L}{2}}\rho^{n-1}(u')^2dx_1}{\frac{P(D_1)}{\epsilon^{n-2}}u^2(-\frac{L}{2})+\frac{P(D_2)}{\epsilon^{n-2}}u^2(\frac{L}{2})+w_{n-2}\int_{-\frac{L}{2}}^{\frac{L}{2}}u^2\rho^{n-2}dx_1}\\\label{ep4r}
&\leq \max_{\beta_0(\epsilon)\cup\{\beta_j\}_{j=1}^{k-1}}\frac{w_{n-1}\int_{-\frac{L}{2}}^{\frac{L}{2}}\rho^{n-1}(\beta_0(\epsilon)\psi'+\sum_{i=1}^{k-1}\beta_iV_i' )^2dx_1}{\frac{P(D_1)}{\epsilon^{n-2}}\beta_0(\epsilon)^2+w_{n-2}\int_{-\frac{L}{2}}^{\frac{L}{2}}\rho^{n-2}(\beta_0(\epsilon)\psi+\sum_{i=1}^{k-1}\beta_iV_i )^2dx_1}.
\end{align} 
In the inequality above, in order to study the cases in full generality,  we must impose that  the first real coefficient $\beta_0$ depends on $\epsilon$, because the boundary conditions of the eigenvalue problem \eqref{ep4} depends on $\epsilon$. We define the following quantity 
\begin{equation*}
A_{\beta}(\epsilon)=\frac{w_{n-1}\int_{-\frac{L}{2}}^{\frac{L}{2}}\rho^{n-1}(\beta_0(\epsilon)\psi'+\sum_{i=1}^{k-1}\beta_iV_i' )^2dx_1}{\frac{P(D_1)}{\epsilon^{n-2}}\beta_0(\epsilon)^2+w_{n-2}\int_{-\frac{L}{2}}^{\frac{L}{2}}\rho^{n-2}(\beta_0(\epsilon)\psi+\sum_{i=1}^{k-1}\beta_iV_i )^2dx_1}.
\end{equation*}
It is easy to check that in order to get the maximum \eqref{ep4r} we must have that $\lim_{\epsilon\rightarrow 0}\frac{\beta_0(\epsilon)^2}{\epsilon^{n-2}}=C<\infty$. Otherwise if $\lim_{\epsilon\rightarrow 0}\frac{\beta_0(\epsilon)}{\epsilon^{n-2}}=\infty$ 
we will have that $\lim_{\epsilon\rightarrow 0}A_{\beta}(\epsilon)=0$, so we don't reach the maximum in this case. 

We conclude that $\beta_0(\epsilon)\sim \beta_0\epsilon^{\frac{n-2}{2}}+o(\epsilon^{\frac{n-2}{2}})$, (if $\beta_0(\epsilon)< \beta_0\epsilon^{\frac{n-2}{2}}+o(\epsilon^{\frac{n-2}{2}})$ we have the same results) from \eqref{ep4r} we obtain that 
\begin{align*}
\lambda_k^{\epsilon}&\leq \max_{\{\beta_j\}_{j=0}^{k-1}}\frac{w_{n-1}\int_{-\frac{L}{2}}^{\frac{L}{2}}\rho^{n-1}((\beta_0\epsilon^{\frac{n-2}{2}}+o(\epsilon^{\frac{n-2}{2}}))\psi'+\sum_{i=1}^{k-1}\beta_iV_i' )^2dx_1}{P(D_1)\beta_0^2+w_{n-2}\int_{-\frac{L}{2}}^{\frac{L}{2}}\rho^{n-2}((\beta_0\epsilon^{\frac{n-2}{2}}+o(\epsilon^{\frac{n-2}{2}}))\psi+\sum_{i=1}^{k-1}\beta_iV_i )^2dx_1}\\
&\leq \max_{\{\beta_j\}_{j=1}^{k-1}}\frac{\sum_{j=1}^{k-1} \beta_j^2w_{n-1}\int_{-\frac{L}{2}}^{\frac{L}{2}}(V_j')^2\rho^{n-1}dx_1}{\sum_{j=1}^{k-1} \beta_j^2w_{n-2}\int_{-\frac{L}{2}}^{\frac{L}{2}}V_j^2\rho^{n-2}dx_1}+o(\epsilon).
\end{align*}
Recalling that $V_1,..,V_{k-1}$ are eigenfunctions of the problem \eqref{ep3}, by the variational characterization of the eigenvalue of the problem \eqref{ep3} we finally conclude that 
\begin{equation*}
\alpha_{k-1}=\max_{\{\beta_j\}_{j=1}^{k-1}}\frac{\sum_{j=1}^{k-1} \beta_j^2w_{n-1}\int_{-\frac{L}{2}}^{\frac{L}{2}}(V_j')^2\rho^{n-1}dx_1}{\sum_{j=1}^{k-1} \beta_j^2w_{n-2}\int_{-\frac{L}{2}}^{\frac{L}{2}}V_j^2\rho^{n-2}dx_1}
\end{equation*}
and this concludes the proof.
\end{proof}
We are ready to prove Theorem \ref{th2}.
\begin{proof}[Proof of Theorem \ref{th2}]
From the second part of Lemma \ref{lemb1}, Lemma 
\ref{bb1} and Lemma \ref{lemba1}, we finally conlclude that for all $k\geq 2$ we have 
\begin{equation*}
\sigma_k^{\epsilon}\sim \alpha_{k-1} \epsilon+o(\epsilon) \quad \text{as} \quad \epsilon\rightarrow 0.
\end{equation*}
We prove  the convergence of the eigenfunctions, showing that $V_k$ must be the  $(k-1)$-th eigenfunction of  problem \eqref{ep2}. From Lemma \ref{lemcn5} we know that $V_k$ satisfies the Dirichlet boundary conditions, it remains to prove the fact that $V_k$ satisfies the eigenvalue equation.
We use the variational formulation of the Steklov problem using the following test function $\phi\in C_c^{\infty}(-\frac{L}{2},\frac{L}{2})$ (we constantly extend $\phi$ in the last variables $x'$), we obtain:
\begin{equation*}
\int_{T_{\epsilon}} \frac{\partial u_k^{\epsilon} }{\partial x_1} \frac{\partial \phi }{\partial x_1} dx=\sigma_k^{\epsilon} \int_{\partial T_{\epsilon}^e}  u_k^{\epsilon} \phi d\mathcal{H}^{n-1}.
\end{equation*}

We make the following change of variable $y_1=x_1$ and $y'=\epsilon x'$ in the right hand side of the variational formulation, using the formula \eqref{surfmeas2} for the surface measure in the right hand side we obtain 
\begin{align*}
\int_{T_1} \frac{\partial v_k^{\epsilon}}{\partial y_1} \frac{\partial \phi }{\partial y_1} dy&=\frac{\sigma_k^{\epsilon}\epsilon^{\frac{n-2}{2}}}{\epsilon}\big ( \int_{-\frac{L}{2}}^{\frac{L}{2}}\int_{\partial B_{\epsilon\rho(x_1)}^{n-1}(x_1)}u_k^{\epsilon}(x_1,\epsilon\tilde{\rho}(x_1)) \phi(x_1)\rho^{n-2}\sqrt{1+\epsilon^2\rho'^2}dx_1d\varphi_1...d\varphi_{n-2}\big )\\
&=\frac{\sigma_k^{\epsilon}}{\epsilon}\big ( \int_{-\frac{L}{2}}^{\frac{L}{2}}\int_{\partial B_{\rho(x_1)}^{n-1}(x_1)}v_k^{\epsilon}(x_1,\tilde{\rho}(x_1)) \phi(x_1)\rho^{n-2}\sqrt{1+\epsilon^2\rho'^2}dx_1d\varphi_1...d\varphi_{n-2}\big ),
\end{align*}
where $\tilde{\rho}(x_1)$ are the spherical coordinates that describes $\partial B_{\rho(x_1)}^{n-1}(x_1)$. 
Now we let $\epsilon$ goes to $0$, recalling that $v_k^{\epsilon} \rightharpoonup \overline{V}_k$  in $H^1(T_1)$ and $\sigma_k^{\epsilon}\sim \alpha_{k-1} \epsilon+o(\epsilon)$  we obtain
\begin{equation*}
\int_{T_1} \frac{\partial \overline{V}_k}{\partial y_1} \frac{\partial \phi }{\partial y_1} dy=\alpha_{k-1} \int_{-\frac{L}{2}}^{\frac{L}{2}} w_{n-2}\rho^{n-2}V_k \phi dx_1.
\end{equation*}
Now integrating by part the left hand side, we finally obtain
\begin{equation*}
-w_{n-1}\int_{-\frac{L}{2}}^{\frac{L}{2}} \frac{d}{dx_1} \big (  \rho^{n-1} \frac{d}{dx_1}  V_k \big ) \phi dx_1=\alpha_{k-1}\int_{-\frac{L}{2}}^{\frac{L}{2}} w_{n-2}\rho^{n-2} V_k \phi dx_1.
\end{equation*}
This relation is true for every test function $\phi\in C_c^{\infty}(-\frac{L}{2},\frac{L}{2})$ so we have 
\begin{equation*}
-w_{n-1}\frac{d}{dx}\big(\rho^{n-1}(x)\frac{d V_k}{dx}(x)\big)=\alpha_{k-1} w_{n-2}\rho^{n-2}(x) V_k(x)  \qquad  x\in \big(-\frac{L}{2},\frac{L}{2} \big),
\end{equation*}
and $V_k(-\frac{L}{2})=V_k(\frac{L}{2})=0$. Using the same argument as at the end of proof of Theorem \ref{th1} we conclude that the result is true for all the sequence $\{ \epsilon_n\}_{n=1}^{\infty}$. This concludes the proof.  
\end{proof}
\section{The case $n\geq 3$ and $k=1$. Proof of Theorem \ref{th3}.}
In this section we prove Theorem \ref{th3}.
\subsection{Convergence of the eigenfunctions}
From Lemma \ref{lemcn1}, we know that 
\begin{align*}
u_1^{\epsilon} &\rightharpoonup c_{i ,1}\quad \text{in} \quad H^1(D_i),\\
u_1^{\epsilon} &\rightarrow c_{i,2} \quad \text{locally uniformly in} \quad D_i,
\end{align*}
We also know that $c_{1 ,1}>0>c_{2,1}$, this information letting us to improve Lemma \ref{lemcn3}.
\begin{lemma}\label{lemcn11} 
Let $n\geq 3$ and let $\Omega_{\epsilon}\subset \mathbb{R}^n$ be a dumbbell shape domain. There exists a constant $C>0$ such that
\begin{equation}\label{asyef1}
\limsup_{\epsilon\rightarrow 0}\frac{\int_{\partial T_{\epsilon}}(u^{\epsilon}_1)^2d\mathcal{H}^{n-1}}{\epsilon^{n-2}}\leq C.
\end{equation}
\end{lemma}
\begin{proof}
By contradiction we suppose 
\begin{equation}\label{contr3}
N\epsilon^{n-2}\leq\int_{\partial T_{\epsilon}}(u^{\epsilon}_1)^2d\mathcal{H}^{n-1} \quad \forall\, N\in \mathbb{N}.
\end{equation}
Using the same argument in the proof of Lemma \ref{lemcn3} we conclude that
\begin{equation*}
\epsilon\int_{\partial T_{\epsilon}^e}(u^{\epsilon}_1)^2d\mathcal{H}^{n-1}\leq\int_{ T_{\epsilon}^e}(u^{\epsilon}_1)^2dx+\epsilon^2  \int_{T_{\epsilon}^e}|\nabla_{x'}u^{\epsilon}_1|^2dx.
\end{equation*}
from this inequality, the fact that $\sigma_1^{\epsilon}\leq C\epsilon^{n-1}$ and inequality \eqref{contr3} we finally have 
\begin{equation*}
\epsilon^{n-1}\frac{N}{2}\leq \int_{ T_{\epsilon}^e}(u^{\epsilon}_1)^2dx \quad \forall\, N\in \mathbb{N}.
\end{equation*}
We know that $c_{1 ,1}>0>c_{2,1}$, repeating all the arguments in the first part of the proof of Lemma \ref{lemcn3}, we obtain that 
\begin{equation*}
CN\epsilon^{n-1}\leq \int_{E_{\epsilon}}|\nabla u_1^{\epsilon}|^2dx \quad \forall\, N\in \mathbb{N},
\end{equation*}
where $E_{\epsilon}$ is the extended tube defined in \eqref{extt}. This is a contradiction with the fact that $\sigma_1^{\epsilon}\leq C\epsilon^{n-1}$.
\end{proof}
We introduce the following function
\begin{equation*}
v_1^{\epsilon}(x_1,x')=u_1^{\epsilon}(x_1,\epsilon x')\quad \forall \, (x_1,x')\in T_1.
\end{equation*}
\begin{lemma}\label{lemcn12}
Let $n\geq 3$ then there exists $\overline{V}_1\in H^1(T_1)$ such that (up to a sub-sequence that we still denote by $v_1^{\epsilon}$) 
\begin{equation*}
v_1^{\epsilon} \rightharpoonup \overline{V}_1 \quad \text{in} \quad H^1(T_1).
\end{equation*}
and $\overline{V}_1$ depends only on the variable $x_1$.
\end{lemma}
\begin{proof}
We start with the bound of $||\nabla v_1^{\epsilon}||_{L^2(T_1)}$ 
\begin{equation*}
\int_{T_1}|\nabla v_1^{\epsilon}|^2 dx\leq \int_{T_1} \Big( \frac{\partial v_1^{\epsilon}}{\partial x_1} \Big )^2+\frac{1}{\epsilon^2} |\nabla_{x'} v_1^{\epsilon} |^2 dx=\frac{1}{\epsilon^{n-1}}\int_{T_{\epsilon}}|\nabla u_1^{\epsilon}|^2 dy\leq C
\end{equation*}
where we did the change of coordinates $y_1=x_1$, $y'=\epsilon x'$ and the last inequality is true because of \eqref{ubs1}.  We want now to bound $||v_1^{\epsilon}||_{L^2(T_1)}$. Following the computations in the proof of Lemma \ref{lemcn3} we obtain 
\begin{align*}
\int_{T_1^e}(v^{\epsilon}_1)^2dx=\frac{1}{\epsilon^{n-1}}\int_{T_{\epsilon}^e}(u^{\epsilon}_1)^2dx&\leq\int_{-\frac{L}{2}}^{\frac{L}{2}}\frac{\int_{B_{\epsilon\rho(x_1)}^{n-1}(x_1)}|\nabla u_1^{\epsilon}|^2d\mathcal{H}^{n-1}+\int_{\partial B_{\epsilon\rho(x_1)}^{n-1}(x_1)}(u_1^{\epsilon})^2d\mathcal{H}^{n-2}}{\epsilon^{n-1} \lambda_1(B_{\epsilon\rho(x_1)}^{n-1}(x_1),1)}\\
&\leq \frac{C}{\epsilon^{n-2}}\big (\int_{T_{\epsilon}^e}|\nabla u_1^{\epsilon}|^2dx+\int_{\partial T_{\epsilon}^e}(u_1^{\epsilon})^2d\mathcal{H}^{n-1} \big )\\
&\leq C,
\end{align*}
where the last inequality is true beacuse of  \eqref{asyef1} and \eqref{ubs1}. 
We conclude that there exist $\overline{V}_1\in H^1(T_1)$ such that (up to a sub-sequence that we still denote by $v_1^{\epsilon}$) 
\begin{equation*}
v_1^{\epsilon} \rightharpoonup \overline{V}_1 \quad \text{in} \quad H^1(T_1).
\end{equation*}

We finish the proof by showing that $\overline{V}_1$ does not depend on $x_i$ for all $i\geq 2$, indeed 
\begin{equation*}
\int_{T_1} \big ( \frac{\partial v_1^{\epsilon} }{\partial x_i} \big )^2dx=\frac{1}{\epsilon^{n-3}} \int_{T_{\epsilon}} \big ( \frac{\partial u_k^{\epsilon} }{\partial x_i} \big )^2dx\leq C \epsilon^2 \rightarrow 0.
\end{equation*} 
\end{proof}

We denote by $V_1$ the restriction to the $x_1$ variable of the function $\overline{V}_1$ and we introduce the extended tube $E_{\epsilon}$ (see \eqref{extt}). In the next Lemma we find the boundary conditions of $V_k$.
\begin{lemma}\label{lemcn13}
Let $n\geq 3$ and let $V_1$ be the restriction to $x_1$ of the limit eigenfunction $\overline{V}_1$ in Lemma \ref{lemcn12} then $V_1$ is continuous and
$$V_1(-\frac{L}{2})=c_{1,1}\quad \mbox{and} \quad V_1(\frac{L}{2})=c_{2,1}.$$
\end{lemma} 
\begin{proof}
The first point is immediate because we know that, if we consider the extended tube $E_1$, $V_1\in H^1((-\frac{L}{2}-\xi,\frac{L}{2}+\xi))$, by classical embedding theorem we have $V_1\in C((-\frac{L}{2}-\xi,\frac{L}{2}+\xi))$.
We prove the second point, we know, from Lemma \ref{lemcn11} that $u_1^{\epsilon}$ locally uniformly converge to $c_{1,1}$ in $D_1$. From this fact we have that 
\begin{equation*}
v_1^{\epsilon}(-\frac{L}{2}-\delta,0)=u_1^{\epsilon}(-\frac{L}{2}-\delta,0)\rightarrow c_{1,1}=V_1(-\frac{L}{2}-\delta)\quad \forall \, \xi\geq \delta>0.
\end{equation*}
From the continuity of $V_1$ we conclude that 
\begin{equation*}
V_1(-\frac{L}{2})=\lim_{\delta\rightarrow 0} V_1(-\frac{L}{2}-\delta)=c_{1,1}.
\end{equation*}
Using the same techniques we obtain also $V_1(\frac{L}{2})=c_{2,1}$. 
\end{proof}
\subsection{Proof of Theorem \ref{th3}}
In this section we prove Theorem \ref{th3}. We will show that the bound from above given in the first part of Lemma \ref{lemb1} is actually the right asymptotics. In particular the following result holds.
\begin{lemma}\label{bb3}Let $n\geq 3$ then, for $\epsilon$ small enough 
\begin{equation}
\sigma_1^{\epsilon}\geq \sigma_1\epsilon^{n-1}+o(\epsilon^{n-1}) 
\end{equation}
where $\sigma_1$ is the positive number such that the following differential equation has a solution:
\begin{equation*}
\begin{cases}
\vspace{0.2cm}
   -w_{n-1}\frac{d}{dx}\big(\rho^{n-1}(x)\frac{d V_1}{dx}(x)\big)=0  \qquad  x\in \big(-\frac{L}{2},\frac{L}{2} \big) \\
   \vspace{0.2cm}
   \rho^{n-1}(-\frac{L}{2})\frac{dV_1}{dx}(-\frac{L}{2})=-\frac{\sigma_1}{w_{n-1}}P(D_1)V_1(-\frac{L}{2}) \\
   \rho^{n-1}(\frac{L}{2})\frac{dV_1}{dx}(\frac{L}{2})=\frac{\sigma_1}{w_{n-1}}P(D_2)V_1(\frac{L}{2}).
\end{cases}
\end{equation*}
\end{lemma}
\begin{proof}
From the variational characterization of the first Steklov eigenfunction we have:
\begin{equation*}
\sigma_1^{\epsilon}=\frac{\int_{\Omega_{\epsilon}}|\nabla u_1^{\epsilon}|^2dx}{\int_{\partial\Omega_{\epsilon}}(u_1^{\epsilon})^2d\mathcal{H}^{n-1}}\geq \epsilon^{n-1}\frac{\int_{T^e_1}|\nabla v_1^{\epsilon}|^2dx}{\int_{\partial D_1\cup \partial D_2 }(u_1^{\epsilon})^2d\mathcal{H}^{n-1}+o(\epsilon)}.
\end{equation*}
From this inequality and the convergence results, in particular Lemmas \ref{lemcn1}, \ref{lemcn12} and \ref{lemcn13} we obtain 
\begin{equation*}
\sigma_1^{\epsilon}\geq \epsilon^{n-1}\frac{w_{n-1}\int_{-\frac{L}{2}}^{\frac{L}{2}}(V_1')^2\rho^{n-1}dx}{P(D_1)V_1^2(-\frac{L}{2})+P(D_2)V_1^2(\frac{L}{2})}+o(\epsilon^{n-1})\geq \epsilon^{n-1}\sigma_1+o(\epsilon^{n-1}),
\end{equation*}
where $\sigma_1>0$ is the positive number such that the following differential equation has a solution:
\begin{equation*}
\begin{cases}
\vspace{0.2cm}
   -w_{n-1}\frac{d}{dx}\big(\rho^{n-1}(x)\frac{d V_1}{dx}(x)\big)=0  \qquad  x\in \big(-\frac{L}{2},\frac{L}{2} \big) \\
   \vspace{0.2cm}
   \rho^{n-1}(-\frac{L}{2})\frac{dV_1}{dx}(-\frac{L}{2})=-\frac{\sigma_1}{w_{n-1}}P(D_1)V_1(-\frac{L}{2}) \\
   \rho^{n-1}(\frac{L}{2})\frac{dV_1}{dx}(\frac{L}{2})=\frac{\sigma_1}{w_{n-1}}P(D_2)V_1(\frac{L}{2}).
\end{cases}
\end{equation*}
\end{proof}
Now we are ready to prove Theorem \ref{th3}
\begin{proof}[Proof of Theorem \ref{th3}]
From the first part of Lemma \ref{lemb1} and Lemma \ref{bb3}, we get that  
\begin{equation*}
\sigma_1^{\epsilon}\sim\sigma_1 \epsilon^{n-1}+o(\epsilon^{n-1}) \quad \text{as} \quad \epsilon\rightarrow 0, 
\end{equation*}
We use the variational formulation of the Steklov problem using the following test function $\phi\in C_c^{\infty}(-\frac{L}{2},\frac{L}{2})$ (we constantly extend $\phi$ in the last variables $x'$), we obtain:
\begin{equation*}
\int_{T_{\epsilon}} \frac{\partial u_1^{\epsilon} }{\partial x_1} \frac{\partial \phi }{\partial x_1} dx=\sigma_k^{\epsilon} \int_{\partial T_{\epsilon}^e}  u_1^{\epsilon} \phi d\mathcal{H}^{n-1}.
\end{equation*}

We perform the following change of variable $y_1=x_1$ and $y'=\epsilon x'$ in the right hand side of the variational formulation, using the formula \eqref{surfmeas2} for the surface measure in the right hand side we obtain 
\begin{align*}
\int_{T_1} \frac{\partial v_1^{\epsilon}}{\partial y_1} \frac{\partial \phi }{\partial y_1} dy&=\frac{\sigma_k^{\epsilon}}{\epsilon}\big ( \int_{-\frac{L}{2}}^{\frac{L}{2}}\int_{\partial B_{\epsilon\rho(x_1)}^{n-1}(x_1)}u_k^{\epsilon}(x_1,\epsilon\tilde{\rho}(x_1)) \phi(x_1)\rho^{n-2}\sqrt{1+\epsilon^2\rho'^2}dx_1d\varphi_1...d\varphi_{n-2}\big )\\
&=\frac{\sigma_k^{\epsilon}}{\epsilon}\big ( \int_{-\frac{L}{2}}^{\frac{L}{2}}\int_{\partial B_{\rho(x_1)}^{n-1}(x_1)}v_k^{\epsilon}(x_1,\tilde{\rho}(x_1)) \phi(x_1)\rho^{n-2}\sqrt{1+\epsilon^2\rho'^2}dx_1d\varphi_1...d\varphi_{n-2}\big ),
\end{align*}
where $\tilde{\rho}(x_1)$ are the spherical coordinates that describes $\partial B_{\rho(x_1)}^{n-1}(x_1)$. 

We let $\epsilon$ goes to $0$, recalling that $v_1^{\epsilon} \rightharpoonup \overline{V}_1$  in $H^1(T_1)$ and the fact that $\sigma_1^{\epsilon}\sim\sigma_1 \epsilon^{n-1}+o(\epsilon^{n-1})$  we obtain
\begin{equation*}
\int_{T_1} \frac{\partial \overline{V}_1 }{\partial y_1} \frac{\partial \phi }{\partial y_1} dy=0.
\end{equation*}
Integrating by parts in the left hand side, we finally obtain
\begin{equation*}
-w_{n-1}\int_{-\frac{L}{2}}^{\frac{L}{2}} \frac{d}{dy} \big (  \rho^{n-1} \frac{d}{dy}  V_1 \big ) \phi dy=0.
\end{equation*}
This relation is true for every test function $\phi\in C_c^{\infty}(-\frac{L}{2},\frac{L}{2})$ so we have that $V_1$ and $\alpha_1$ must have to satisfy the following differential equation
\begin{equation}\label{sen1}
-\frac{d}{dx}\big(\rho(x)^{n-1}\frac{d V_1}{dx}(x)\big)=0  \qquad  x\in \big(-\frac{L}{2},\frac{L}{2} \big).
\end{equation}

In order to find the boundary conditions for this equation we use the variational formulation with a test function $\psi$ defined on all $\Omega_{\epsilon}$ and which depends only on $x_1$,
\begin{equation*}
\int_{\Omega_{\epsilon}} \frac{\partial u_1^{\epsilon} }{\partial x_1} \frac{\partial \psi }{\partial x_1} dx=\sigma_1^{\epsilon} \int_{\partial \Omega_{\epsilon}}  u_1^{\epsilon} \psi d\mathcal{H}^{n-1}.
\end{equation*}
We recall that $u_1^{\epsilon}$ uniformly converge to $c_{1,1}$ in $D_1$ and to $c_{2,1}$ in $D_2$, so we extend the functions $v_1^{\epsilon}$ to be equal to $u_1^{\epsilon}$ in $D_1$ and the same for $D_2$. From Lemma \ref{lemcn13} we have that $v_1^{\epsilon}\rightarrow c_{1,1}=V_1 \big (-\frac{L}{2}\big ) $ in $D_1$ and  $v_1^{\epsilon}\rightarrow c_{2,1}=V_1 \big (\frac{L}{2}\big ) $ in $D_2$.
We repeat all the computations that we did above and we obtain
\begin{equation*}
w_{n-1}\int_{-\frac{L}{2}}^{\frac{L}{2}}\rho^{n-1} \frac{d \overline{V}_1 }{dx_1} \frac{d \psi }{d x_1} dx_1=\sigma_1\big (  V_1 \big (-\frac{L}{2}\big )\int_{\partial D_1}\psi  d\mathcal{H}^{n-1} +   V_1 \big ( \frac{L}{2}\big )\int_{\partial D_2} \psi  d\mathcal{H}^{n-1} \big ).
\end{equation*}

Integrating by parts in the left hand side and recalling \eqref{sen1}, we finally obtain
\begin{eqnarray*}
\rho^{n-1}(\frac{L}{2})\frac{dV_1}{dx}(\frac{L}{2})\psi(\frac{L}{2})-\rho^{n-1}(-\frac{L}{2})\frac{dV_1}{dx}(-\frac{L}{2})\psi(-\frac{L}{2})=\\
=\sigma_1\big (  V_1 \big (-\frac{L}{2}\big )\int_{\partial D_1}\psi  d\mathcal{H}^{n-1} +   V_1 \big ( \frac{L}{2}\big )\int_{\partial D_2} \psi d\mathcal{H}^{n-1} \big ).
\end{eqnarray*}
 We choose the test function such that $\psi=1$ in $D_1$ and $\psi=0$ in $D_2$ and we deduce the first boundary condition
\begin{equation*}
 \rho^{n-1}(-\frac{L}{2})\frac{dV_1}{dx}(-\frac{L}{2})=-\frac{\sigma_1}{w_{n-1}}P(D_1)V_1(-\frac{L}{2}).
\end{equation*} 
Similarly if we choose the test function such that $\psi=0$ in $D_1$ and $\psi=1$ in $D_2$ we get the second boundary condition 
\begin{equation*}
 \rho^{n-1}(\frac{L}{2})\frac{dV_1}{dx}(\frac{L}{2})=\frac{\sigma_1}{w_{n-1}}P(D_2)V_1(\frac{L}{2}).
\end{equation*}

We finally obtain the following differential equation
\begin{equation*}
\begin{cases}
\vspace{0.2cm}
   -w_{n-1}\frac{d}{dx}\big(\rho^{n-1}(x)\frac{d V_1}{dx}(x)\big)=0  \qquad  x\in \big(-\frac{L}{2},\frac{L}{2} \big) \\
   \vspace{0.2cm}
   \rho^{n-1}(-\frac{L}{2})\frac{dV_1}{dx}(-\frac{L}{2})=-\frac{\sigma_1}{w_{n-1}}P(D_1)V_1(-\frac{L}{2}) \\
   \rho^{n-1}(\frac{L}{2})\frac{dV_1}{dx}(\frac{L}{2})=\frac{\sigma_1}{w_{n-1}}P(D_2)V_1(\frac{L}{2}).
\end{cases}
\end{equation*}
Using the same argument as at the end of proof of Theorem \ref{th1} we conclude that the result is true for all the sequence $\{ \epsilon_n\}_{n=1}^{\infty}$. This concludes the proof.  
\end{proof}
\section{Application: counter-example to a Spectral Inequality}
We consider the Neumann eigenvalue problem
\begin{equation*}
\begin{cases}
     -\Delta v_k=\mu_k v_k\qquad  &\Omega \\
      \partial_{\nu} v_k=0 \qquad &\partial\Omega.
\end{cases}
\end{equation*}
During the writing of the paper \cite{GHL} came the following question: is it true that the inequality
\begin{equation*}
\mu_1|\Omega|\geq \sigma_1P(\Omega),
\end{equation*}
holds for any plane domains?
For several domains like balls, annulus, rectangles, convex sets  with a ratio between the inradius
and circumradius large enough, this inequality turns out to be true. 

Nevertheless, the results of \cite{BN20,GKL20} implicitly show that the inequality can not be true in general, its failure coming either from highly oscillating boundaries or from the presence of a large number of small holes. 

Our aim in this section is to provide another counter-example which is simply connected and do not have an oscillating boundary, for which the reverse inequality holds:
\begin{equation*}
\mu_1|\Omega|< \sigma_1P(\Omega).
\end{equation*}
Consider the domain $\Omega_{\epsilon}\subset \mathbb{R}^2$ constructed as follow, $D_1$ and $D_2$ are two balls such that $|D_1|=|D_2|=1$ and $\rho=1$ constantly.

\begin{figure}[H]
\centering
\tdplotsetmaincoords{0}{0}
\begin{tikzpicture}[tdplot_main_coords, scale=0.5]
\draw (-3,1) arc (10:350:3);
\draw [dashed](-3,-0.15) -- (-3,1);
\draw (-3,-0.15) -- (3,-0.15);
\draw (-3,1) -- (3,1);
\draw (3,1) arc (170:-170:3);
\draw [dashed](3,-0.15) -- (3,1);
\draw (-6,1) node {$D_1$};
\draw (6,1) node {$D_2$};
\draw (0,0.3) node {$T_{\epsilon}$};
\draw [dashed] (-6,0.3) -- (-8.9,0.3);
\draw (-7.3,-0.3) node {$r=1/\sqrt{\pi}$};
\filldraw  (-6,0.3) circle (2pt);
\draw [dashed] (6,0.3) -- (8.9,0.3);
\draw (7.3,-0.3) node {$r=1/\sqrt{\pi}$};
\filldraw  (6,0.3) circle (2pt);
\end{tikzpicture}
\caption{Dumbbell shape domain in $n=2$, with $\rho\equiv 1$.}
\label{fig2}
\end{figure}
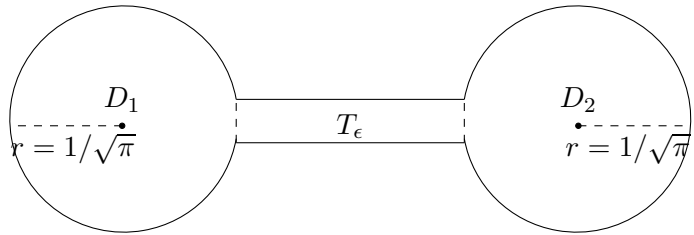

One can compute the eigenvalues of problem \eqref{ep1}, which becomes
\begin{equation*}
\begin{cases}
\vspace{0.2cm}
  - \frac{d^2 V_k}{dx^2}(x)=\alpha_k V_k(x)  \qquad  x\in \big(-\frac{L}{2},\frac{L}{2} \big) \\
\vspace{0.2cm}
     \frac{dV_k}{dx}(-\frac{L}{2})=-\frac{\alpha_k}{2}P(D_1)V_k(-\frac{L}{2}) \\
      \frac{dV_k}{dx}(\frac{L}{2})=\frac{\alpha_k}{2}P(D_2)V_k(\frac{L}{2}).
\end{cases}
\end{equation*}
The general solution has the following form $V_k=A\cos(w_kx)+B\sin(w_kx)$ where $w_k^2=\alpha_k$ and the boundary conditions give us the following equations for the unknowns $A$ and $B$,
 
\begin{equation*}
\begin{cases}
\vspace{0.2cm}
A(\sin(w_k\frac{L}{2})+w_k\frac{P(D_1)}{2}\cos(w_k\frac{L}{2}))+B(\cos(w_k\frac{L}{2})-w_k\frac{P(D_1)}{2}\sin(w_k\frac{L}{2}))=0 \\

A(-\sin(w_k\frac{L}{2})-w_k\frac{P(D_2)}{2}\cos(w_k\frac{L}{2}))+B(\cos(w_k\frac{L}{2})-w_k\frac{P(D_2)}{2}\sin(w_k\frac{L}{2}))=0.
\end{cases}
\end{equation*}

In order to have non trivial solutions the determinant of this $2\times 2$ system must be equal to $0$, so we obtain that $w_k$ must be satisfies the following transcendental equation
\begin{equation}\label{eqdeterm}
\cot(w_kL)=\frac{w_k^2P(D_1)P(D_2)-4}{2w_k(P(D_1)+P(D_2))}.
\end{equation}

We know that $\sigma_1^{\epsilon}\sim \alpha_1 \epsilon$, where $\alpha_1=w_1^2$ where $w_1$ is the first value for which the equation \eqref{eqdeterm} holds. 

By choosing, in the variational formulation, a test function which is constant on each disk and affine in the tube, we can prove that $\mu_1^{\epsilon}\leq \frac{4}{L} \epsilon$, so we conclude that 
\begin{equation*}
\frac{|\Omega_{\epsilon}|}{P(\Omega_{\epsilon})}\mu_1^{\epsilon}\leq \frac{4}{(2\sqrt{\pi}+L)L}\epsilon
+o(\epsilon).
\end{equation*}

If we prove that there exists $L>0$ such that:
\begin{equation*}
 \frac{4}{(2\sqrt{\pi}+L)L}<\alpha_1
\end{equation*}
we conclude that there exists $\overline{\epsilon}$ such that 
\begin{equation*}
\mu_1^{\overline{\epsilon}}|\Omega_{\overline{\epsilon}}|< \sigma_1^{\overline{\epsilon}}P(\Omega_{\overline{\epsilon}}).
\end{equation*}

We introduce the following function
\begin{equation*}
f(w)=\cot(wL)-\frac{w^2\pi-1}{2\sqrt{\pi}w},
\end{equation*}
an easy computation shows that
\begin{equation*}
0<f(x)\quad \forall \quad 0<x<\frac{3\pi}{4L} \Leftrightarrow \quad L>\frac{3}{4}(\sqrt{2}+1)\pi^{\frac{3}{2}}.
\end{equation*} 

So we conclude that for all $L$ such that $ L>\frac{3}{4}(\sqrt{2}+1)\pi^{\frac{3}{2}}$ we have 
\begin{equation*}
\alpha_1\geq \frac{9\pi^2}{16L^2}>\frac{4}{(2\sqrt{\pi}+L)L}.
\end{equation*}
providing the desired counter-example.

\bigskip\noindent
{\bf Acknowledgements}: 
This work was partially supported by the project ANR-18-CE40-0013 SHAPO financed by the French Agence Nationale de la Recherche (ANR).

\bibliographystyle{mybst}
\bibliography{References}

\providecommand{\bysame}{\leavevmode\hbox to3em{\hrulefill}\thinspace}
\providecommand{\MR}{\relax\ifhmode\unskip\space\fi MR }
\providecommand{\MRhref}[2]{%
  \href{http://www.ams.org/mathscinet-getitem?mr=#1}{#2}
}
\providecommand{\href}[2]{#2}
\begin{thebibliography}{10}

\bibitem{Ar95}
Jos\'{e}~M. Arrieta, \emph{Neumann eigenvalue problems on exterior
  perturbations of the domain}, J. Differential Equations \textbf{118} (1995),
  no.~1, 54--103. \MR{1329403}

\bibitem{ACL06}
Jos\'{e}~M. Arrieta, Alexandre~N. Carvalho, and German Lozada-Cruz,
  \emph{Dynamics in dumbbell domains. {I}. {C}ontinuity of the set of
  equilibria}, J. Differential Equations \textbf{231} (2006), no.~2, 551--597.
  \MR{2287897}

\bibitem{Be73}
J.~Thomas Beale, \emph{Scattering frequencies of reasonators}, Comm. Pure Appl.
  Math. \textbf{26} (1973), 549--563. \MR{352730}

\bibitem{BuBu}
Dorin Bucur and Giuseppe Buttazzo, Variational methods in shape optimization
  problems, Progress in Nonlinear Differential Equations and their
  Applications, vol.~65, Birkh\"{a}user Boston, Inc., Boston, MA, 2005.
  \MR{2150214}

\bibitem{BN20}
Dorin Bucur and Micka\"{e}l Nahon, \emph{Stability and instability issues of
  the Weinstock inequality}, 2020, Arxiv arXiv:2004.07784.

\bibitem{Da06}
Daniel Daners, \emph{A {F}aber-{K}rahn inequality for {R}obin problems in any
  space dimension}, Math. Ann. \textbf{335} (2006), no.~4, 767--785.
  \MR{2232016}

\bibitem{Fa90}
Qing Fang, \emph{Asymptotic behavior and domain-dependency of solutions to a
  class of reaction-diffusion systems with large diffusion coefficients},
  Hiroshima Math. J. \textbf{20} (1990), no.~3, 549--571. \MR{1083427}

\bibitem{GS05}
Tiziana Giorgi and Robert~G. Smits, \emph{Monotonicity results for the
  principal eigenvalue of the generalized {R}obin problem}, Illinois J. Math.
  \textbf{49} (2005), no.~4, 1133--1143. \MR{2210355}

\bibitem{GHL}
Alexandre Girouard, Antoine Henrot, and Jean Lagac\'e, \emph{From Steklov to
  Neumann via homogenisation}, preprint https://arxiv.org/abs/1906.09638
  (2020), 34 pages.

\bibitem{GKL20}
Alexandre Girouard, Mikhail Karpukhin, and Jean Lagac\'e, \emph{Sharp
  isoperimetric upper bounds for planar Steklov eigenvalues}, 2020, Arxiv
  arXiv:2004.10784.

\bibitem{GP10}
Alexandre Girouard and Iosif Polterovich, \emph{Shape optimization for low
  {N}eumann and {S}teklov eigenvalues}, Math. Methods Appl. Sci. \textbf{33}
  (2010), no.~4, 501--516. \MR{2641628}

\bibitem{HV84}
Jack~K. Hale and Jos\'{e} Vegas, \emph{A nonlinear parabolic equation with
  varying domain}, Arch. Rational Mech. Anal. \textbf{86} (1984), no.~2,
  99--123. \MR{751304}

\bibitem{HPi}
Antoine Henrot and Michel Pierre, Shape variation and optimization, EMS Tracts
  in Mathematics, vol.~28, European Mathematical Society (EMS), Z\"{u}rich,
  2018, A geometrical analysis. \MR{3791463}

\bibitem{Ji188}
Shuichi Jimbo, \emph{Singular perturbation of domains and semilinear elliptic
  equation}, J. Fac. Sci. Univ. Tokyo Sect. IA Math. \textbf{35} (1988), no.~1,
  27--76. \MR{931442}

\bibitem{Ji288}
Shuichi Jimbo, \emph{Singular perturbation of domains and the semilinear
  elliptic equation. {II}}, J. Differential Equations \textbf{75} (1988),
  no.~2, 264--289. \MR{961156}

\bibitem{Ji89}
Shuichi Jimbo, \emph{The singularly perturbed domain and the characterization
  for the eigenfunctions with {N}eumann boundary condition}, J. Differential
  Equations \textbf{77} (1989), no.~2, 322--350. \MR{983298}

\bibitem{Ji93}
Shuichi Jimbo, \emph{Perturbation formula of eigenvalues in a singularly
  perturbed domain}, J. Math. Soc. Japan \textbf{45} (1993), no.~2, 339--356.
  \MR{1206658}

\bibitem{Ji04}
Shuichi Jimbo, \emph{Singular perturbation of domains and semilinear elliptic
  equations. {III}}, Hokkaido Math. J. \textbf{33} (2004), no.~1, 11--45.
  \MR{2034806}

\bibitem{JM92}
Shuichi Jimbo and Yoshihisa Morita, \emph{Remarks on the behavior of certain
  eigenvalues on a singularly perturbed domain with several thin channels},
  Comm. Partial Differential Equations \textbf{17} (1992), no.~3-4, 523--552.
  \MR{1163435}

\bibitem{Mo90}
Yoshihisa Morita, \emph{Reaction-diffusion systems in nonconvex domains:
  invariant manifold and reduced form}, J. Dynam. Differential Equations
  \textbf{2} (1990), no.~1, 69--115. \MR{1041198}

\bibitem{Na12}
S.~A. Nazarov, \emph{Asymptotic behavior of the eigenvalues of the {S}teklov
  problem on a junction of domains of different limiting dimensions}, Zh.
  Vychisl. Mat. Mat. Fiz. \textbf{52} (2012), no.~11, 2033--2049. \MR{3247705}

\bibitem{Na14}
S.~A. Nazarov, \emph{Asymptotic expansions of eigenvalues of the {S}teklov
  problem in singularly perturbed domains}, Algebra i Analiz \textbf{26}
  (2014), no.~2, 119--184. \MR{3242037}

\bibitem{vBF}
Joachim von Below and Gilles Fran\c{c}ois, \emph{Spectral asymptotics for the
  {L}aplacian under an eigenvalue dependent boundary condition}, Bull. Belg.
  Math. Soc. Simon Stevin \textbf{12} (2005), no.~4, 505--519. \MR{2205994}

\end{thebibliography}

\end{document}